\documentclass[11pt,twoside]{amsart}
\usepackage{latexsym,amssymb,amsmath}
\usepackage[all]{xy}
\usepackage{tikz,pgfplots}
\usepackage{extpfeil}
\usepackage{afterpage}
\usepackage{float}
\usepackage{hyperref}

\makeatletter  
\@namedef{subjclassname@2020}{%
\textup{2020} Mathematics Subject Classification}
\makeatother

\numberwithin{equation}{section}
\def\cR{\mathcal{R}}
\def\N{\mathbb{N}}
\newtheorem{theorem}{Theorem}[section]
\newtheorem{lemma}[theorem]{Lemma}
\newtheorem{proposition}[theorem]{Proposition}
\newtheorem{corollary}[theorem]{Corollary}

\theoremstyle{definition}
\newtheorem{definition}[theorem]{Definition} 
\newtheorem{procedure}[theorem]{Procedure} 
\newtheorem{algorithm}[theorem]{Algorithm}
\newtheorem{remark}[theorem]{Remark}
\newtheorem{example}[theorem]{Example}

\begin{document}

\title[Rees algebras of filtrations]
{Rees algebras of 
filtrations of covering polyhedra and integral closure of powers of monomial ideals} 
\thanks{The first author was supported by a scholarship from CONACYT,
Mexico. The third author was supported by SNI, Mexico.}

\author[G. Grisalde]{Gonzalo Grisalde}
\address{
Departamento de
Matem\'aticas\\
Centro de Investigaci\'on y de Estudios
Avanzados del
IPN\\
Apartado Postal
14--740 \\
07000 Mexico City, Mexico
}
\email{gjgrisalde@math.cinvestav.mx}

\author[A. Seceleanu]{Alexandra Seceleanu}
\address{Department of Mathematics,
203 Avery Hall, Lincoln, NE 68588.
}
\email{aseceleanu@unl.edu}

\author[R. H. Villarreal]{Rafael H. Villarreal}
\address{
Departamento de
Matem\'aticas\\
Centro de Investigaci\'on y de Estudios
Avanzados del
IPN\\
Apartado Postal
14--740 \\
07000 Mexico City, Mexico
}
\email{vila@math.cinvestav.mx}

\subjclass[2020]{Primary 13C70; Secondary 13F20, 13F55, 05E40, 13A30, 13B22} 

\dedicatory{Dedicated to Professor J\"urgen Herzog on the occasion of his
$80$th birthday}  

\begin{abstract} 
The aims of this work are to study Rees algebras of filtrations of
monomial ideals associated to
covering polyhedra of rational matrices with non-negative entries
and non-zero columns using
combinatorial optimization and integer programming, and to study
powers of monomial 
ideals and their integral
closures using irreducible decompositions and polyhedral
geometry. We 
study the Waldschmidt constant and the ic-resurgence of the filtration 
associated to a covering polyhedron and show how to compute these
constants using linear programming. Then
we show a lower bound for the ic-resurgence of the ideal of covers of a
graph and prove that the lower bound is attained when the graph is
perfect. We also show lower bounds for the ic-resurgence of the edge
ideal of a graph and give an algorithm to compute the asymptotic
resurgence of squarefree monomial ideals. 
A classification of when Newton's polyhedron is the irreducible 
polyhedron is presented using integral closure. 
\end{abstract}

\maketitle 

\section{Introduction}\label{section-intro}
Let $S=K[t_1,\ldots,t_s]$ be a polynomial ring over a field $K$. The monomials of $S$ are denoted
by $t^a:=t_1^{a_1}\cdots t_s^{a_s}$, $a=(a_1,\dots,a_s)$ in $\mathbb{N}^s$, where
$\mathbb{N}=\{0,1,\ldots\}$. Let $I$ be a monomial ideal of $S$ minimally generated by
the set of monomials $G(I):=\{t^{v_1},\ldots,t^{v_q}\}$.  
The \textit{incidence matrix} of $I$, denoted by $A$,
is the $s\times q$ matrix with column vectors $v_1,\ldots,v_q$. The
$\textit{covering polyhedron}$ of $I$, denoted by
$\mathcal{Q}(I)$, is the rational polyhedron
$$
\mathcal{Q}(I):=\{x\vert\, x\geq 0;\,xA\geq 1\},
$$
where $1=(1,\ldots,1)$. The \textit{Newton
polyhedron} of $I$, denoted ${\rm NP}(I)$, is the 
integral polyhedron 
\begin{equation}\label{NP-def}
{\rm NP}(I)=\mathbb{R}_+^s+{\rm 
conv}(v_1,\ldots,v_q),
\end{equation}
where $\mathbb{R}_+=\{\lambda\in\mathbb{R}\vert\, \lambda\geq
0\}$. 
This polyhedron is the convex hull of the exponent vectors
of the monomials in $I$ \cite[p.~141]{Eisen} and is equal to $\{x\in\mathbb{R}^s\vert\,
x\geq 0;\,xB\geq 1\}$ for some rational matrix $B$ with non-negative
entries (Proposition~\ref{np-qa}). The \textit{integral closure} of
 $I^n$ can be described as 
\begin{equation}\label{jun21-21}
\overline{I^n}=
(\{t^a\vert\, a/n\in{\rm NP}(I)\})
\end{equation}
for all $n\geq 1$ \cite[Proposition~3.5(a)]{reesclu}. 
If $I$ is squarefree, the $n$-th \textit{symbolic
power} of $I$ is
given by 
\begin{equation}\label{jun21-21-1}
I^{(n)}=(\{t^a\vert\, a/n\in\mathcal{Q}(I^\vee)\}),
\end{equation}
where $I^\vee$ is the Alexander dual of $I$ \cite[p.~78]{reesclu}.
The covering polyhedron $\mathcal{Q}(I^\vee)$ is called the \textit{symbolic polyhedron} of
$I$ and is denoted by ${\rm SP}(I)$ \cite[p.~50]{Cooper-symbolic}. 
If $I = \overline{I}$, $I$ is said to be 
{\em complete}. If all the powers $I^n$ are
complete, $I$  is said to be \textit{normal}. 

We now introduce a central notion that generalizes the Newton polyhedron and
the covering polyhedron of a monomial ideal. 
A \textit{covering polyhedron} is a rational polyhedron of the 
form 
$$
\mathcal{Q}(C):=\{x\vert\, x\geq 0;\,xC\geq 1\},
$$
for some $s\times m$ rational matrix $C$ with entries in
$\mathbb{Q}_+=\{\lambda\in\mathbb{Q}\vert\, \lambda\geq 0\}$ and
with non-zero columns. 
A covering polyhedron is of blocking
type in the sense of \cite[p.~114]{Schr} and the Newton polyhedron is
a covering polyhedron (Lemma~\ref{blocking-type}).

To a covering polyhedron $\mathcal{Q}(C)$, we associate the 
decreasing sequence $\mathcal{F}=\{I_n\}_{n=0}^\infty$ of monomial ideals of
$S$ given by 
\begin{equation}\label{sequence-def}
I_n:=(\{t^a\vert\, a/n\in\mathcal{Q}(C)\}),\ n\geq 1,\,\ I_0=S.
\end{equation}
\quad The sequence $\mathcal{F}$ satisfies $\overline{I_n}=I_n$ for all
$n\geq 1$ and is a \textit{filtration} of ideals of $S$, 
that is, $I_{n+1}\subset I_n$, $I_0=S$, and $I_{k}I_n\subset
I_{k+n}$ for all $k,n\in\mathbb{N}$ (Lemma~\ref{filtration}). In certain cases the filtration $\mathcal{F}$ is
\textit{strict}, that is, $I_{n+1}\subsetneq I_n$ for all $n\geq 0$ (Lemma~\ref{filtration}). 
We call $\mathcal{Q}(C)$ the \textit{covering
polyhedron} of $C$. The filtration
$\mathcal{F}$ 
is called the \textit{filtration} associated 
to $\mathcal{Q}(C)$. If $I$ is a monomial ideal, the filtration
associated to the Newton polyhedron ${\rm NP}(I)$ of $I$ is the filtration 
$\{\overline{I^n}\}_{n=0}^\infty$ of integral closure of powers of
$I$ (Eq.~\eqref{jun21-21}), and if $I$ is a squarefree monomial ideal,
the filtration associated to $\mathcal{Q}(I^\vee)$ is the filtration
$\{I^{(n)}\}_{n=0}^\infty$ of
symbolic powers of $I$ (Eq.~\eqref{jun21-21-1}).  

The \textit{initial degree} of a monomial ideal $I$, denoted by
$\alpha(I)$, is the least degree of a minimal generator of $I$. We
associate to $\mathcal{F}$ the function 
$\alpha_\mathcal{F}\colon\mathbb{N}_+\rightarrow\mathbb{N}_+$ given by
\begin{equation}\label{af-eq}
\alpha_\mathcal{F}(n):=\alpha(I_n)=\min\{\deg(t^a)\vert\, t^a\in
I_n\},
\end{equation}
where $\mathbb{N}_+=\{1,2,\ldots\}$. The \textit{Waldschmidt constant} of
$\mathcal{F}$, denoted $\widehat{\alpha}(\mathcal{F})$, is defined as  
\begin{equation}\label{waldschmidt-eq}
\widehat{\alpha}(\mathcal{F}):=\lim_{n\rightarrow\infty}{\alpha_{\mathcal{F}}(n)}/{n}.
\end{equation}
\quad This limit exists and is equal to the infimum of 
$\alpha_\mathcal{F}(n)/n$, $n\geq 1$ (Lemma~\ref{subadditive-g}). One
of our main results shows that $\widehat{\alpha}(\mathcal{F})$ can be
expressed as the optimal value of a linear program. 

\noindent \textbf{Theorem~\ref{schrijver-number-lp}.}\textit{ 
Let $\mathcal{F}=\{I_n\}_{n=0}^\infty$ be the filtration associated to the covering polyhedron 
$\mathcal{Q}(C)$.  
If $y=(y_1,\ldots,y_s)$, then the linear program
\begin{enumerate}
\item[] {\rm minimize} $y_1+\cdots+y_s$
\item[] {\rm subject to}
\item[] $yC\geq 1$ {\rm and }$y\geq 0$
\end{enumerate}
has an optimal value equal to $\widehat{\alpha}(\mathcal{F})$, which is 
attained at a rational vertex $\beta$ of $\mathcal{Q}(C)$.
}

If $\mathcal{Q}=\mathcal{Q}(C)$ and $\alpha(\mathcal{Q})$ is the minimum of all
$|v|=\sum_{i=1}^sv_i$ with $v=(v_1,\ldots,v_s)$ a vertex of
$\mathcal{Q}$, then
$\alpha(\mathcal{Q})=\widehat{\alpha}(\mathcal{F})$
(Corollary~\ref{schrijver-constant}). This result was shown in  
\cite[Corollary~6.3]{Cooper-symbolic} when $\mathcal{Q}$ is the
symbolic polyhedron of a monomial ideal $I$ and $\mathcal{F}$ is the
filtration of symbolic powers of $I$. A covering polyhedron is
\textit{integral} if and only if it has only integral vertices \cite[p.~232]{Schr}. We show that
$\alpha_{\mathcal{F}}(1)\geq \widehat{\alpha}(\mathcal{F})$, with equality if 
$\mathcal{Q}(C)$ is integral (Proposition~\ref{integral-qc}). If $I$
is a complete monomial ideal generated by monomials of degree $d$, 
then the limit
of $\alpha(\overline{I^n})/{n}$ when $n$ goes to infinity is $d$ (Corollary~\ref{jun12-21}).

Let $\mathcal C$ be a \textit{clutter} with vertex 
set $V(\mathcal{C})=\{t_1,\ldots,t_s\}$, that is, $\mathcal C$ is a 
family of subsets $E(\mathcal{C})$ of $V(\mathcal{C})$, called edges,
none of which is contained in
another. The {\it edge
ideal\/} of $\mathcal{C}$, denoted by $I(\mathcal{C})$,
is the ideal of $S$ generated by all  
monomials $t_e=\prod_{t_i\in e}t_i$ such 
that $e\in E(\mathcal{C})$. Any squarefree monomial ideal is the edge
ideal of a clutter. The support of a monomial $t^a$, denoted 
${\rm supp}(t^a)$, is the set of all $t_i$ that occur in $t^a$. The
\textit{ideal of covers} of $\mathcal{C}$, denoted
$I_c(\mathcal{C})$, is the ideal of $S$ generated by
all squarefree monomials whose support is a minimal vertex cover of
$\mathcal{C}$ \cite[p.~221]{monalg-rev}. In the context of Stanley--Reisner theory of simplicial
complexes, $I_c(\mathcal{C})$ is called the
Alexander dual of $I=I(\mathcal{C})$ and is denoted by
$I^\vee$ \cite[pp.~17--18]{Herzog-Hibi-book}.

If $I$ is a squarefree monomial ideal and
$\mathcal{F}=\{I^{(n)}\}_{n=0}^\infty$ is the filtration 
associated to $\mathcal{Q}(I^\vee)$, then $\widehat{\alpha}(\mathcal{F})$ is the
Waldschmidt constant of $I$ and is denoted by $\widehat{\alpha}(I)$
\cite{Cooper-symbolic,Francisco-TAMS}. As a consequence of
Theorem~\ref{schrijver-number-lp} we recover the fact that
$\widehat{\alpha}(I)$ is the value of the optimal solution of a linear
program \cite[Theorem~3.2]{Waldschmidt-Bocci-etal}. If $\mathcal{Q}(I)$ is integral, using 
\cite[Theorem~1.8]{cornu-book}, we recover the
formulas $\widehat{\alpha}(I)=\alpha(I)$ and
$\widehat{\alpha}(I^\vee)=\alpha(I^\vee)$
\cite[Theorem~4.3]{Seceleanu-packing} (Corollary~\ref{jun12-21-1}).

We classify 
the family of ideals that satisfy 
the equality
``$\overline{I_1^n}=I_n$ for all $n\geq 1$''
  when 
$\mathcal{F}=\{I_n\}_{n=0}^\infty$ is the filtration of 
a covering polyhedron (Proposition~\ref{equality-in}), and recover the
classification of Fulkersonian clutters given in \cite{clutters,Trung}
(Corollary~\ref{ntf-char}, cf. Remark~\ref{aug7-21}). Then we also classify 
the family of ideals that satisfy 
the equality
``$I_1^n=I_n$ for all $n\geq 1$'' 
(Proposition~\ref{equality-in}). 
If $\mathcal{F}$ comes from the covering polyhedron of the Alexander 
dual of the edge ideal $I=I(\mathcal{C})$ of a clutter $\mathcal{C}$,
this equality becomes 
``$I^n=I^{(n)}$ for all $n\geq 1$'' and, by
\cite[Corollary~3.14]{clutters} or \cite[Theorem 1.4]{hhtz}, the
equality holds if and only if $\mathcal{C}$ has the max-flow min-cut property
(Definition~\ref{mfmc-def}). The clutter $\mathcal{C}$ has the
max-flow min-cut property if and only if $\mathcal{Q}(I)$ is integral 
and $I$ is normal \cite[Theorem 3.4]{clutters}. Our
classification is a generalization of these facts (Corollary~\ref{ntf-char}).

The equality
between symbolic and ordinary powers of squarefree monomial ideals 
was related to a conjecture of Conforti and Cornu\'ejols
\cite[Conjecture~1.6]{cornu-book} 
on the max-flow min-cut property of clutters in 
\cite[Theorem~4.6, Conjecture~4.18]{reesclu} and \cite[Conjecture~3.10]{clutters}. 

We now turn our attention on the Rees algebra of a filtration 
$\mathcal{F}=\{I_n\}_{n=0}^\infty$ associated to a covering polyhedron
$\mathcal{Q}(C)$. 
The \textit{Rees algebra of the filtration}
$\mathcal{F}$, denoted $\mathcal{R}(\mathcal{F})$, is given by
$$
\mathcal{R}(\mathcal{F}):=S\textstyle\bigoplus I_1z\textstyle\bigoplus\cdots\textstyle\bigoplus
I_nz^n\textstyle\bigoplus\cdots\subset S[z],
$$
where $z$ is a new variable. To show some of the algebraic properties
of $\mathcal{R}(\mathcal{F})$ we need to extend the notion of a Simis
cone \cite{normali} to covering polyhedra.   
The \textit{Simis cone} of $\mathcal{Q}=\mathcal{Q}(C)$, denoted ${\rm SC}(\mathcal{Q})$, is the rational
polyhedral cone in $\mathbb{R}^{s+1}$ given by 
\begin{equation}\label{simis-cone-def}
{\rm SC}(\mathcal{Q})=\{x\in\mathbb{R}^{s+1}\vert\, x\geq 0;\, \langle x,(c_i,-1)\rangle\geq 0\,
\forall\,i\}, 
\end{equation}
where $c_i$ is the $i$-th column of $C$ and 
$\langle\ ,\, \rangle$ is the ordinary inner product in
$\mathbb{R}^{s+1}$. 

The Hilbert basis $\mathcal{H}$ of ${\rm SC}(\mathcal{Q})$ is
the set of all integral 
vectors $0\neq \alpha\in{\rm SC}(\mathcal{Q})$ such
that $\alpha$ is not the sum of two other non-zero integral vectors in
${\rm SC}(\mathcal{Q})$ \cite{Schr1}.
A polyhedron containing no lines is called  
{\it pointed}. Note that a covering polyhedron is always pointed. 
A face of dimension $1$ of a pointed polyhedral cone is
called an {\it extreme ray}.  

The vertices of $\mathcal{Q}$ are related to the extreme rays of ${\rm SC}(\mathcal{Q})$.
If $V(\mathcal{Q})=\{\beta_1,\ldots,\beta_r\}$ is the vertex set of
$\mathcal{Q}$, we show that the Simis cone is generated by the set   
$$\mathcal{B}'=\{e_1,\dots,e_s,
(\beta_1,1),\ldots,(\beta_r,1)\},
$$
where $e_i$ is the $i$-th unit vector in $\mathbb{R}^{s+1}$, 
and prove that $\mathcal{B}'$ is a set of representatives for the
extreme rays of ${\rm SC}(\mathcal{Q})$ (Proposition~\ref{scv}). 
The cone generated by $\mathcal{B}'$ is denoted by
$\mathbb{R}_+\mathcal{B}'$. In
general $\mathcal{B}'$ is not the Hilbert basis of  
${\rm SC}(\mathcal{Q})$ because the $\beta_i$'s might not be
integral.  

To compute the generators of the symbolic Rees algebra of a monomial
ideal $I$ one
can use the algorithm in the proof 
of \cite[Theorem 1.1]{cover-algebras} and 
\cite[Proposition~4]{cm-oriented-trees}. Another of our main results
shows that the Rees algebra of $\mathcal{F}$ is a 
normal $K$-algebra minimally generated by the monomials that
correspond to the points of the Hilbert basis 
$\mathcal{H}$ of ${\rm SC}(\mathcal{Q})$
(Theorem~\ref{hilbert-basis-filtration}). 

If $\mathcal{F}$ is the filtration of a covering polyhedron,
then there exists an integer $k\geq 1$ such that 
$\widehat{\alpha}(\mathcal{F})={\alpha_\mathcal{F}(nk)}/{nk}$ for all
$n\geq 1$ (Proposition~\ref{jun13-21}). This result was shown in
\cite[Corollary~6.2]{Cooper-symbolic}
when $\mathcal{F}$ is the filtration of symbolic powers of a monomial
ideal. 

The resurgence and asymptotic resurgence of ideals were introduced 
in \cite{resurgence,asymptotic-resurgence}. The resurgence of an ideal
relative to the integral closure filtration was introduced in
\cite{Francisco-TAMS}. We define similar notions for filtrations of
ideals. 
Let $\mathcal{F}=\{I_n\}_{n=0}^\infty$ be the filtration associated to a
covering polyhedron $\mathcal{Q}(C)$ and let
$\mathcal{F}'=\{J_n\}_{n=0}^\infty$ be another filtration of ideals of
$S$.   
We define the \textit{resurgence} and {\em asymptotic resurgence of the
filtration $\mathcal{F}$ 
relative to $\mathcal{F}'$} to be 
\begin{align*}
&\rho(\mathcal{F}, \mathcal{F}'):=\left.\sup \left\{{m}/{r}
\ \right|\, 
I_{m}\not\subset J_{r}\right\},\\ 
&\ \ \ \ \ \widehat{\rho}(\mathcal{F}, \mathcal{F}'):=\left.\sup
\left \{ {m}/{r}\ 
\right|\, 
I_{mt}\not\subset J_{rt} \text{ for all } t\gg 0\right\},\
\text{respectively}.
\end{align*}

The following are interesting special cases of the resurgence and asymptotic
resurgence of $\mathcal{F}$ relative to a filtration $\mathcal{F}'$:

\begin{enumerate} 
\item[] If $\mathcal{F}'=\{I_1^n\}_{n=0}^\infty$, we denote
$\rho(\mathcal{F}, \mathcal{F}')$ and $\widehat{\rho}(\mathcal{F},
\mathcal{F}')$ by $\rho(\mathcal{F})$ and
$\widehat{\rho}(\mathcal{F})$, respectively. We call $\rho(\mathcal{F})$ and
$\widehat{\rho}(\mathcal{F})$ the \textit{resurgence} and
\textit{asymptotic resurgence} 
of $\mathcal{F}$. 
\smallskip
\item[] If $\mathcal{F}'=\{\overline{I_1^n}\}_{n=0}^\infty$, we denote
$\rho(\mathcal{F}, \mathcal{F}')$ and $\widehat{\rho}(\mathcal{F},
\mathcal{F}')$ by $\rho_{ic}(\mathcal{F})$ and
$\widehat{\rho}_{ic}(\mathcal{F})$, respectively. We call $\rho_{ic}(\mathcal{F})$ and
$\widehat{\rho}_{ic}(\mathcal{F})$ the \textit{ic-resurgence} and
\textit{ic-asymptotic resurgence} 
of $\mathcal{F}$.
\end{enumerate} 

If $\mathcal{F}$ is a strict filtration, then 
$\widehat{\rho}(\mathcal{F})=\widehat{\rho}_{ic}(\mathcal{F})={\rho}_{ic}(\mathcal{F})$
and this number is finite 
(Section~\ref{section-comparison}, Lemma~\ref{bound-ic}). This result
was inspired by the study of asymptotic resurgence of ideals using
integral closures 
of Dipasquale, Francisco, Mermin and Schweig and specially by their result
that for any ideal $I$ of $S$ one has
$\widehat{\rho}(I)=\widehat{\rho}_{ic}(I)={\rho}_{ic}(I)$ 
\cite[Corollary~4.14]{Francisco-TAMS}.  

Let $I$ be a squarefree monomial ideal of $S$ and
let $\mathcal{F}=\{I^{(n)}\}_{n=0}^\infty$ be the filtration
associated to $\mathcal{Q}(I^\vee)$. 
The resurgence and ic-resurgence of $\mathcal{F}$ are denoted by
$\rho(I)$ and $\rho_{ic}(I)$, respectively \cite{resurgence,Francisco-TAMS}.  
In this
case it is known that the computation of $\rho_{ic}(I)$ can be
reduced to linear programming \cite[Section~2]{Francisco-TAMS} (cf.
\cite{DiPasquale-Drabkin}). The
main result of Section~\ref{section-computing} shows that the ic-resurgence
$\rho_{ic}(\mathcal{F})$ of a strict filtration
$\mathcal{F}$ of a covering polyhedron $\mathcal{Q}(C)$ can be computed using 
linear programming. 
We give an algorithm, implemented in \textit{Normaliz} \cite{normaliz2} and \textit{Macaulay}$2$
\cite{mac2}, 
to compute the asymptotic resurgence of a
squarefree monomial ideal (Procedure~\ref{bowtie-procedure},
Algorithm~\ref{AS-code}). 

To state our result we need some notation. Let $c_1,\ldots,c_m$ be
the columns of the matrix $C$, let $B$ be a matrix with 
entries in $\mathbb{Q}_+$ such that the Newton
polyhedron of $I_1$ is $\mathcal{Q}(B)$, 
let $\beta_1,\ldots,\beta_k$ be the columns of $B$, and let $n_i$ be a
positive integer such that $n_i\beta_i$ is integral for
all $i$. 

We come to another of our main results.

\noindent \textbf{Theorem~\ref{lp-resurgence-formula}.}\textit{
For each $1\leq j\leq k$, let $\rho_j$ be the optimal value of the
following linear program with variables $y_1,\ldots,y_{s+3}$. If $\mathcal{F}$ is strict, then 
$\rho_{ic}(\mathcal{F})=\max\{\rho_j\}_{j=1}^k$.
\begin{align}
&\text{maximize }\ \ g_j(y)=y_{s+1}&&&&\nonumber\\
\quad&\text{subject to }\ \langle(y_1,\ldots,y_s),c_i\rangle-y_{s+1}\geq
0,\ i=1,\ldots,m,\, y_{s+1}\geq y_{s+3} &&&&\nonumber \\
&\quad \quad\quad\quad\quad y_i\geq 0,\, i=1,\ldots,s,\, y_{s+3}\geq 0&&&&\nonumber \\
&\quad \quad\quad\quad\quad  n_jy_{s+2}-\langle(y_1,\ldots,y_s),n_j\beta_j\rangle
\geq y_{s+3},\, y_{s+2}=1.&&&&\nonumber 
\end{align}
}
\quad Let $G$ be a graph and let $I(G)$ be its edge ideal. From
\cite[Theorem~3.12, Corollary~4.14]{Francisco-TAMS} and 
\cite[Theorem~6.7(i)]{Waldschmidt-Bocci-etal}, $\rho_{ic}(I(G))$ is
${2(\omega(G)-1)}/{\omega(G)}$ if $G$ is perfect, where $\omega(G)$ is the
clique number of $G$. The next result shows a similar formula for the
ideal of covers $I_c(G)$ of $G$. 
If $G$ is perfect, then $\rho(I_c(G))$ is
equal to $\rho_{ic}(I_c(G))$ because $I_c(G)$ is normal 
\cite[Theorem~2.10]{perfect}.

\noindent \textbf{Theorem~\ref{jun14-21}.}\textit{ 
$\rho_{ic}(I_c(G))\geq{2(\omega(G)-1)}/{\omega(G)}$ 
with equality if $G$ is perfect. 
}

If $G$ is a graph, we show that  
$\rho_{ic}(I(G))\geq {2\alpha_0(H)}/{|V(H)|}
$ 
for any induced
subgraph $H$ of $G$, where $\alpha_0(H)$ is the covering number
of $H$ (Proposition~\ref{jun14-21-1}),
and if $G$ is non-bipartite, we show
that $\widehat{\alpha}(I(G))\leq\widehat{\alpha}(I_c(G))$
(Proposition~\ref{jun14-21-2}). 

Let $I$ be a monomial ideal of $S$. 
We will show that the covering polyhedron $\mathcal{Q}(I)$ of $I$ is related
to the irreducible decomposition of $I$ that we now introduce. Recall
that an ideal $L$ of $S$ is called {\it irreducible} if 
$L$ cannot be written as an intersection of two ideals of $S$ that
properly contain $L$. 
Given $b=(b_1,\ldots,b_s)$ in $\mathbb{N}^s\setminus\{0\}$, we set
$\mathfrak{q}_b:=(\{t_i^{b_i}\vert\, b_i\geq 1\})$ and
$b^{-1}:=\sum_{b_i\geq 1}b_i^{-1}e_i$. According to
\cite[Theorems~6.1.16 and 6.1.17]{monalg-rev}, the ideal $I$ has a
\textit{unique irreducible decomposition}:
\begin{equation}\label{jun4-21}
I=\mathfrak{q}_{1}\textstyle\bigcap\cdots\bigcap\mathfrak{q}_{m},
\end{equation}
where each $\mathfrak{q}_{i}$ is an irreducible monomial ideal 
of the form $\mathfrak{q}_i=\mathfrak{q}_{\alpha_i}$ for some
$\alpha_i\in\mathbb{N}^s\setminus\{0\}$, and $I\neq\textstyle\bigcap_{i\neq j}\mathfrak{q}_{i}$ for
$j=1,\ldots,m$. The ideals $\mathfrak{q}_{1},\ldots,\mathfrak{q}_{m}$
are 
the {\it irreducible components\/} of $I$. 

Let $B$ be the matrix with column vectors
$\alpha_1^{-1},\ldots,\alpha_m^{-1}$. The covering polyhedron
$\mathcal{Q}(B)$ of $B$ is called the \textit{irreducible polyhedron} of $I$
 and is denoted by ${\rm IP}(I)$ \cite{Seceleanu-convex-bodies}
 (Example~\ref{filtration3}). 

Since irreducible ideals are
primary, the irreducible decomposition of $I$ is a 
primary decomposition of $I$.  The irreducible decomposition of $I$
is irredundant, that is, $I\neq\textstyle\bigcap_{i\neq j}\mathfrak{q}_{i}$ for
$j=1,\ldots,m$ but it is not necessarily
a minimal primary decomposition, that is, $\mathfrak{q}_i$ and
$\mathfrak{q}_j$ 
could have the same radical for
some $i\neq j$. For edge ideals of
weighted oriented graphs and for squarefree monomial ideals, their irreducible
decompositions are minimal \cite{WOG,monalg-rev}.

The next result shows that under some conditions the irreducible
components of a monomial
ideal are related to the vertices of its covering polyhedron
(cf. Remark~\ref{jun26-21}).

\noindent \textbf{Theorem~\ref{irreducible-deco-qA}.}\textit{ Let $I$
be a monomial ideal of $S$,
let
$I=\bigcap_{i=1}^m\mathfrak{q}_i$ be its
irreducible decomposition, and let $\mathcal{Q}(I)$ be the covering 
polyhedron of $I$. The following hold.
\begin{enumerate}
\item[(a)] If $\mathfrak{q}_k=(t_1^{b_1},\ldots,t_r^{b_r})$,
$b_\ell\geq 1$ for all $\ell$, and
${\rm rad}(\mathfrak{q}_j)\not\subset{\rm rad}(\mathfrak{q}_k)$ for $j\neq
k$, then the vector $b^{-1}:=\sum_{i=1}^rb_i^{-1}e_i$ is a
vertex of $\mathcal{Q}(I)$.
\item[(b)] If $I$ has no embedded associated primes and 
${\rm rad}(\mathfrak{q}_j)\neq{\rm rad}(\mathfrak{q}_i)$ for $j\neq
i$, then there are $\alpha_1,\ldots,\alpha_m$ in
$\mathbb{N}^s\setminus\{0\}$ such
that $\mathfrak{q}_i=\mathfrak{q}_{\alpha_i}$ and $\alpha_i^{-1}$ is a
vertex of $\mathcal{Q}(I)$ for $i=1,\ldots,m$.
\end{enumerate}
}

Let $\mathfrak{q}$ be a primary monomial ideal of
$S$. We show that ${\rm NP}(\mathfrak{q})={\rm IP}(\mathfrak{q})$ if and only
if $\mathfrak{q}$ is irreducible (Proposition~\ref{jun15-21}). If
$\mathfrak{q}$ is irreducible, we classify when $\mathfrak{q}$ is
normal and prove that $\mathfrak{q}$ is normal if and only if $\mathfrak{q}$
is complete (Proposition~\ref{normal-irreducible-ideals}).  
In polynomial rings
in two variables any complete ideal is normal by a result of Zariski 
\cite[Appendix 5]{ZS}. 

For monomial ideals we classify when the Newton polyhedron is the irreducible polyhedron
using integral closure: 

\noindent \textbf{Theorem~\ref{NP-IP-char}.}\textit{ 
Let $I$ be a monomial ideal of $S$ and let 
$I=\mathfrak{q}_1\bigcap\cdots\bigcap\mathfrak{q}_m$ be the
irreducible decomposition of $I$. Then ${\rm NP}(I)={\rm IP}(I)$ if
and only if
$\overline{I^n}=\overline{\mathfrak{q}_1^n}\bigcap\cdots\bigcap\overline{\mathfrak{q}_m^n}$
for all $n\geq 1$.
}

For a certain interesting class of monomial ideals, 
we classify when the Newton polyhedron is the irreducible polyhedron
using integral closure and symbolic powers (Theorem~\ref{np=ip-char}).

In Section~\ref{examples-section} we present examples illustrating 
our results. Then in Appendix~\ref{Appendix} we give the procedures
for \textit{Normaliz} \cite{normaliz2}, \textit{PORTA} \cite{porta},   
and \textit{Macaulay}$2$
\cite{mac2}, that are used in the 
examples. 

For unexplained
terminology and additional information,  we refer to 
\cite{mc8,GoNi,huneke-swanson-book,SchenzelFiltrations,Vas1,bookthree} for the theory of Rees algebras,
filtrations and integral closure, \cite{Herzog-Hibi-book,monalg-rev} for the theory of
edge ideals and monomial ideals, and \cite{korte,Schr,Schr2} for combinatorial optimization and integer
programing.  

\section{Preliminaries}\label{section-prelim} 
In this section we introduce Rees algebras of three special
filtrations of ideals and
some results from commutative algebra and polyhedral geometry. To
avoid repetitions, throughout this paper we continue to employ 
the notations and definitions used in Section~\ref{section-intro}. 

\begin{proposition}\cite[Proposition~6.1.7]{monalg-rev}\label{primary-monomial} 
A monomial ideal
$\mathfrak{q}$ of $S$ is 
a primary ideal if and only if, up to permutation of the 
variables, it has the form:
\begin{equation*}
\mathfrak{q}=(t_1^{v_1},\ldots,t_r^{v_r},t^{v_{r+1}},\ldots,t^{v_q}),
\end{equation*}
where $v_i\geq 1$ for $i=1,\ldots,r$ and $\bigcup_{i=r+1}^q{\rm
supp}(t^{v_i})\subset
\{t_1,\ldots,t_r\}$. 
\end{proposition}

\begin{definition}\label{symbolic-power-def}\rm 
Let $I$ be an ideal of $S$ and 
let $\mathfrak{p}_1,\ldots,{\mathfrak p}_r$ be  
the minimal primes of $I$. Given an integer $n\geq 1$, we define 
the $n$-th \textit{symbolic power} of 
$I$ to be the ideal 
$$
I^{(n)}:=\textstyle\bigcap_{i=1}^r
(I^nS_{\mathfrak{p}_i}\textstyle\bigcap
S).
$$
\end{definition}

\begin{lemma}{\rm(\cite[Lemma~2]{cm-oriented-trees},
\cite[Lemma~3.1]{cover-algebras})} 
\label{anoth-one-char-spow-general} 
Let $I$ be a monomial ideal of $S$. If 
$\mathfrak{I}_1,\ldots,\mathfrak{I}_r$ are the primary components
associated to the minimal primes 
of $I$, then 
$$
I^{(n)}={\mathfrak I}_1^{n}\textstyle\bigcap\cdots\bigcap {\mathfrak I}_r^{n}\ 
\mbox{ for all }\ n\geq 1.
$$
\end{lemma}

Let $I$ be a monomial ideal 
of $S$. The \textit{Rees
algebra} of $I$, denoted $\mathcal{R}(I)$, is the Rees algebra of
the filtration $\{I^n\}_{n=0}^\infty$ of powers of $I$ and 
the \textit{symbolic Rees algebra} of
$I$, denoted $\mathcal{R}_s(I)$, is the Rees algebra of the 
filtration $\{I^{(n)}\}_{n=0}^\infty$ of symbolic powers of $I$.
It is well known \cite[p.~168]{Vas1} that the integral 
closure $\overline{\mathcal{R}(I)}$ of $\mathcal{R}(I)$ is the Rees
algebra of the filtration $\{\overline{I^n}\}_{n=0}^\infty$ of
integral closure of powers of $I$. 
Thus, $\mathcal{R}(I)$ 
is normal if and only if $I$ is normal. 

\begin{lemma}\cite[p.~169]{Vas1}\label{icd}
If $I$ is a monomial ideal of $S$ and $n\in\mathbb{N}_+$, then 
\begin{equation*}
\overline{I^n}=(\{t^a\in S\mid (t^a)^{p}\in I^{pn}
\mbox{ for some }p\geq 1\}).
\end{equation*}
\end{lemma}

Given $a\in {\mathbb R}^s\setminus\{0\}$  and 
$c\in {\mathbb R}$, the \textit{affine hyperplane} $H_{(a,c)}$ 
and the \textit{positive closed halfspace} $H^+_{(a,c)}$ 
bounded by $H_{(a,c)}$ are defined as
\[
H_{(a,c)}:=\{x\in{\mathbb R}^s\vert\, \langle x,a\rangle=c\}\ \mbox{ and
}\ H^+_{(a,c)}:=\{x\in{\mathbb R}^s\vert\, 
\langle x,a\rangle\geq c\}.
\]
\quad If $c=0$, $H_a$ will denote $H_{(a,c)}$ and $H_a^+$ will denote
$H^+_{(a,c)}$. 
If $a$ and $c$ are rational, $H^+_{(a,c)}$ 
is called a \textit{rational closed halfspace}. A \textit{rational polyhedron}
is a subset of ${\mathbb R}^s$ which is the 
intersection of a finite number of rational closed halfspaces of 
$\mathbb{R}^s$.

Let $\Gamma$ be a subset of $\mathbb{R}^s$, the cone 
generated by $\Gamma$, 
denoted ${\mathbb R}_+\Gamma$, is the set of all 
linear combinations of $\Gamma$ with coefficients in $\mathbb{R}_+$.
A subset $\mathcal{Q}$ of $\mathbb{R}^s$ is a 
rational polyhedron if and only if
$\mathcal{Q}=\mathcal{P}+\mathbb{R}_+\Gamma$,  
where $\mathcal{P}$ is the convex hull ${\rm conv}(\mathcal{A})$ of a
finite set $\mathcal{A}$ of rational
points and $\mathbb{R}_+\Gamma$ is a cone generated by a finite
set $\Gamma$ of rational points
\cite[Corollary~7.1b]{Schr}. The
computer programs \textit{Normaliz}
\cite{normaliz2} and  \textit{PORTA} \cite{porta} will be used to switch between these 
two representations.

\begin{lemma}\cite[p.~114]{Schr}\label{blocking-type} Let $\mathcal{Q}$ be a
rational polyhedron. The following conditions are equivalent. 
\begin{enumerate}
\item[(a)] $\mathcal{Q}\subset\mathbb{R}_+^s$ and if $y\geq x$ with
$x\in\mathcal{Q}$, implies $y\in\mathcal{Q}$.
\item[(b)] $\mathcal{Q}=\mathbb{R}_+^s+{\rm
conv}(\alpha_1,\ldots,\alpha_r)$ for some $\alpha_1,\ldots,\alpha_r$
in $\mathbb{Q}_+^s$.
\item[(c)] $\mathcal{Q}=\{x\vert\, x\geq 0;\,xD\geq 1\}$ for some
rational matrix $D$ with entries in $\mathbb{Q}_+$.
\end{enumerate}
\end{lemma}

\begin{proposition}\label{cpr}
Let $\mathcal{Q}(C)$ be a covering polyhedron and let
$\{\beta_1,\ldots,\beta_r\}$ be its vertex set. Then
$$
\mathcal{Q}(C)=\mathbb{R}_+^s+{\rm
conv}(\beta_1,\ldots,\beta_r).
$$
\end{proposition}

\begin{proof} As $\mathcal{Q}(C)$ contains no lines, by
the finite basis theorem (see \cite[Theorem~4.1.3]{webster} and its
proof), there are $\gamma_1,\ldots,\gamma_p\in\mathbb{Q}_+^s$ such that 
$$
\mathcal{Q}(C)=\mathbb{R}_+
\{\gamma_1,\ldots,\gamma_p\}+{\rm conv}(\beta_1,\ldots,\beta_r),
$$
where $\beta_1,\ldots,\beta_r$ are the vertices of $\mathcal{Q}(C)$.
According to \cite[p. 100, Eq.~(5)(iv)]{Schr}, the 
polyhedral cone $\mathbb{R}_+\{\gamma_1,\ldots,\gamma_p\}$ is equal to
 ${\rm char.cone}(\mathcal{Q}(C))$, the characteristic cone of
$\mathcal{Q}(C)$. Then, it is 
not hard to see that ${\rm char.cone}(\mathcal{Q}(C))$ is equal to $\mathbb{R}_+^s$.
\end{proof}

\begin{proposition}\label{np-qa} Let
$I$ be a monomial ideal of $S$, let $u_1,\ldots,u_r$ be 
the vertices of $\mathcal{Q}(I)$, and let $B$ be the matrix with
column vectors $u_1,\ldots,u_r$. The following hold.
\begin{enumerate}
\item[(a)] \cite[Proposition~3.5(a)]{reesclu}
$\overline{I^n}=(\{t^a\vert\, 
a/n\in{\rm NP}(I)\})$. 
\item[(b)] \cite[Proposition~3.5(b)]{reesclu} ${\rm NP}(I)=\mathcal{Q}(B)=\{x\vert\, x\geq 0;\,
xB\geq 1\}$. 
\item[(c)] If $I=(t^{v_1},\ldots,t^{v_q})$, then the vertices of ${\rm NP}(I)$ are
contained in $\{v_1,\ldots,v_q\}$. 
\end{enumerate}
\end{proposition}
\begin{proof} (c): Since ${\rm NP}(I)=\mathbb{R}_+^s+{\rm
conv}(v_1,\ldots,v_q)$, by \cite[Propositions 1.1.36 and
1.1.39]{monalg-rev}, the vertices of ${\rm NP}(I)$ are contained 
in the set $\{v_1,\ldots,v_q\}$. 
\end{proof}

\begin{lemma}\label{change-of-coefficients}
Let $\mathcal{B}=\{\beta_1,\ldots,\beta_r\}$ be a set of non-zero rational vectors 
in $\mathbb{R}_+^s$. Then 
\begin{align*}
&(\mathbb{R}_+^s+{\rm conv}(\mathcal{B}))\textstyle\bigcap\mathbb{Q}_+^s=
\mathbb{Q}_+^s+{\rm conv}_\mathbb{Q}(\mathcal{B}).
\end{align*}
\end{lemma}

\begin{proof} The inclusion ``$\supset$'' is clear. To show the 
inclusion ``$\subset$'' take $x$ in the intersection of $\mathbb{Q}_+^s$ 
and $\mathbb{R}_+^s+{\rm conv}(\mathcal{B})$. Consider the set 
$\Gamma=\{e_i\}_{i=1}^s\textstyle\bigcup\{(\beta_i,1)\}_{i=1}^r$ of rational
vectors in $\mathbb{R}^{s+1}$. Note that $(x,1)$ is in
$\mathbb{R}_+\Gamma$. 
Then, using Farkas's lemma
 \cite[Theorem~1.1.25]{monalg-rev}, we obtain that $(x,1)$ is in 
 $\mathbb{Q}_+\Gamma$, the cone generated by $\Gamma$ over $\mathbb{Q}$. It
 follows that $x$ is in $\mathbb{Q}_+^s+{\rm
 conv}_\mathbb{Q}(\mathcal{B})$.
\end{proof}

\section{Rees algebras of filtrations of covering
polyhedra}\label{section-rees}
In this section we study filtrations of covering polyhedra and their
Waldschmidt constants and Rees algebras. For use below $\langle\
,\, \rangle$ denotes the standard inner
product on Euclidean space. 

\begin{lemma}\label{filtration} 
Let $\mathcal{Q}(C)$ be the covering polyhedron of a matrix
$C=(c_{i,j})$ and let $\mathcal{F}=\{I_n\}_{n=0}^\infty$ be the
sequence of ideals associated to $\mathcal{Q}(C)$ defined in
Eq.~\eqref{sequence-def}.  
The following hold.
\begin{enumerate}
\item[(a)] $\overline{I_n}=I_n$ for all $n\geq 1$, and $\mathcal{F}$ is a filtration of $S$, that is,
$I_{k}I_n\subset I_{k+n}$ and $I_{n+1}\subset I_n$ for all $k$ and
$n$ in $\mathbb{N}$. In particular
$I_1^n\subset\overline{I_1^n}\subset I_n$ for all $n\geq 1$.
\item[(b)] If $c_{i,j}\leq 1$ for all $i,j$ or $\mathcal{Q}(C)$ has at
least one integral vertex, then $\mathcal{F}$ is a strict filtration, that is,
$I_{n+1}\subsetneq I_n$ for all $n\geq 0$.
\end{enumerate}
\end{lemma}

\begin{proof} (a): Let $c_1,\ldots,c_m$ be the column vectors of the
matrix $C$. First we show the equality $\overline{I_n}=I_n$ for all $n\geq 1$.
Clearly $I_n\subset\overline{I_n}$. To show the other inclusion take
$t^a\in\overline{I_n}$. Then, by Lemma~\ref{icd}, there is $p\in\mathbb{N}_+$ such that
$(t^a)^p\in (I_n)^{p}$. Hence 
$$pa=\epsilon+\lambda_1w_1+\cdots+\lambda_rw_r,$$
where $\epsilon\in\mathbb{N}^s$, $t^{w_i}\in I_n$ 
and $\lambda_i\in\mathbb{N}$ for all $i$, and
$\sum_{i=1}^r\lambda_i=p$. Therefore
$$
\frac{a}{n}=\frac{\epsilon}{np}+\left(\frac{\lambda_1}{p}\right)\left(\frac{w_1}{n}\right)
+\cdots+
\left(\frac{\lambda_r}{p}\right)\left(\frac{w_r}{n}\right),
$$
where $w_i/n\in\mathcal{Q}(C)$ for all $i$. Hence, since
$(a/n)-(\epsilon/{np})$ is a convex combination of $w_1/n,\ldots,w_r/n$, we get 
that ${a}/{n}\in\mathcal{Q}(C)$, that is, $t^a\in I_n$. Next we show
the inclusion $I_{k}I_n\subset I_{k+n}$. Take $t^a\in I_k$ and
$t^b\in I_n$, that is, $\langle 
a/k,c_i\rangle\geq 1$ and $\langle b/n,c_i\rangle\geq 1$ for all $i$.
Then
$$
\left\langle \frac{a+b}{k+n},c_i\right\rangle=\frac{\left\langle
a+b,c_i\right\rangle}{k+n}=\frac{\left\langle
a,c_i\right\rangle+\left\langle
b,c_i\right\rangle}{k+n}\geq\frac{k+n}{k+n}=1
$$
for all $i$. Thus, $t^at^b\in I_{k+n}$. Finally we show the inclusion
$I_{n+1}\subset I_n$. We may assume $n\geq 1$. Take $t^a\in I_{n+1}$, that is, $\langle
a/(n+1),c_i\rangle\geq 1$ for all $i$. Then 
$$
\left\langle{a}/{n},c_i\right\rangle={\left\langle
a,c_i\right\rangle}/{n}\geq{\left\langle
a,c_i\right\rangle}/{(n+1)}\geq{(n+1)}/{(n+1)}=1
$$
for all $i$. Thus, $t^a\in I_n$. 

(b): Assume that $c_{i,j}\leq 1$ for all $i,j$. Pick a minimal generator $t^a$,
$a=(a_1,\ldots,a_s)$, of the monomial ideal $I_{n+1}$. Then,
$\langle{a}/(n+1),c_i\rangle\geq 1$ for $i=1,\ldots,m$ and
$a_i\in\mathbb{N}$ for $i=1,\ldots,s$. There is $k$ such that 
$a_k\geq 1$. Then
$$
\langle{a}-e_k,c_i\rangle=\langle{a},c_i\rangle-\langle
e_k,c_i\rangle\geq n+1-c_{k,i}\geq n,
$$
for $i=1,\ldots,m$, and $t^{a-e_k}$ is in $I_n$.
Note that $t^{a-e_k}$ cannot be in $I_{n+1}$ because
$t^a=t_kt^{a-e_k}$ and $t^a\in G(I_{n+1})$. Thus, $I_{n+1}\subsetneq I_n$. Now assume that $a$ is
an integral vertex of $\mathcal{Q}(C)$. Then
$\langle{a},c_i\rangle=1$ for some $i$
\cite[Corollary~1.1.47]{monalg-rev}. Assume that $I_n\subset
I_{n+1}$. As $(na)/n\in\mathcal{Q}(C)$, one has $t^{na}\in
I_n$, and consequently $t^{na}\in
I_{n+1}$, that is, $na/(n+1)\in\mathcal{Q}(C)$. Hence
$$
{n}/{(n+1)}=({n}/({n+1}))\langle{a},c_i\rangle=
\left\langle{na}/{(n+1)},c_i\right\rangle\geq 1,
$$
a contradiction. Thus, $I_n\not\subset I_{n+1}$ and
$I_{n+1}\subsetneq I_n$. 
\end{proof}

\begin{lemma}\cite[Lemma~A.4.1]{Scheinerman}\label{subadditivity} 
If $g\colon\mathbb{N}_+\rightarrow\mathbb{R}$ is a subadditive
function, that is, for all $n_1$, $n_2$, we have
$g(n_1 + n_2)\leq g(n_1) + g(n_2)$ and $g(n)\geq 0$ for all $n$, then
$\lim_{n\rightarrow\infty}{g(n)}/{n}$
exists and is equal to the infimum of $g(n)/n $ $(n\in\mathbb{N}_+)$.
\end{lemma}

\begin{lemma}\label{subadditive-g} Let $\mathcal{F}=\{I_n\}_{n=0}^\infty$ be the filtration
of a covering polyhedron and let $\alpha_\mathcal{F}$ be the function
in Eq.~\eqref{af-eq}. Then, $\lim_{n\rightarrow\infty}{\alpha_{\mathcal{F}}(n)}/{n}$ exists and is
the  infimum of $\alpha_\mathcal{F}(n)/n $ $(n\in\mathbb{N}_+)$.
\end{lemma}

\begin{proof} By Lemma~\ref{subadditivity} we need only show that 
$\alpha_\mathcal{F}$ is subadditive. Pick $t^a\in I_{n_1}$ and $t^b\in I_{n_2}$
 such that $\deg(t^a)=\alpha_\mathcal{F}(n_1)$ and
$\deg(t^b)=\alpha_\mathcal{F}(n_2)$. By Lemma~\ref{filtration}, one has $t^at^b\in
I_{n_1+n_2}$. Hence
$$
\alpha_\mathcal{F}(n_1+n_2)\leq\deg(t^at^b)=\deg(t^a)+\deg(t^b)
=\alpha_\mathcal{F}(n_1)+\alpha_\mathcal{F}(n_2).
$$
\quad Note that $\alpha_\mathcal{F}(n)\geq 1$ because $1\notin I_n$ for
$n\geq 1$. 
\end{proof}

\begin{theorem}\label{schrijver-number-lp} 
Let $\mathcal{F}=\{I_n\}_{n=0}^\infty$ be the filtration of the covering polyhedron 
$\mathcal{Q}(C)$.  
If $\widehat{\alpha}(\mathcal{F})$ is the Waldschmidt constant 
defined in Eq.~\eqref{waldschmidt-eq} and
$y=(y_1,\ldots,y_s)$, then the linear program
\begin{enumerate}
\item[] {\rm minimize} $y_1+\cdots+y_s$

\item[] {\rm subject to}

\item[] $yC\geq 1$ {\rm and }$y\geq 0$
\end{enumerate}
has an optimal value equal to $\widehat{\alpha}(\mathcal{F})$, which is 
attained at a rational vertex  $\beta$ of $\mathcal{Q}(C)$.
\end{theorem}

\begin{proof} There is a vertex $\beta=(\beta_1,\ldots,\beta_s)$ of
$\mathcal{Q}(C)$ such that $|\beta|:=\sum_{i=1}^s\beta_i$ is the optimal
value of the linear program \cite[Proposition~1.1.41]{monalg-rev}. In
particular $|\beta|\leq |\beta'|$ for any other vertex $\beta'$ of 
$\mathcal{Q}(C)$. As $C$ is a rational matrix, $\beta$ has
non-negative rational entries. Then, there is an integer $n\geq 1$
such that $n\beta$ is integral. Writing $\beta=(n\beta)/n$, 
we obtain that $(n\beta)/n$ is in $\mathcal{Q}(C)$, that is,
$t^{n\beta}\in I_n$. Thus, $\deg(t^{n\beta})=n|\beta|\geq
\alpha_\mathcal{F}(n)$, and consequently $|\beta|\geq \alpha_\mathcal{F}(n)/n$.
By Lemma~\ref{subadditive-g}, $\widehat{\alpha}(\mathcal{F})$ is the infimum of
all $\alpha_\mathcal{F}(p)/p$ $(p\in\mathbb{N}_+)$. Thus, 
$|\beta|\geq\widehat{\alpha}(\mathcal{F})$. To show equality we proceed by contradiction
assuming $|\beta|>\widehat{\alpha}(\mathcal{F})$. By Lemma~\ref{subadditive-g},
the sequence $\{\alpha_\mathcal{F}(p)/p\}_{p=1}^{\infty}$ converges to
$\widehat{\alpha}(\mathcal{F})$. Hence, there is $n\geq 1$ such that
$\alpha_\mathcal{F}(n)/n<|\beta|$. Pick $t^a\in I_n$
such that $\deg(t^a)=\alpha_\mathcal{F}(n)$. As $a/n$ is in
$\mathcal{Q}(C)$ and $|\beta|$ is the optimal value of the linear
program, we get 
$$  
|\beta|\leq{|a|}/{n}={\deg(t^a)}/{n}={\alpha_\mathcal{F}(n)}/{n}<|\beta|,
$$
a contradiction. Thus, $|\beta|=\widehat{\alpha}(\mathcal{F})$.
\end{proof}

\begin{corollary}\label{schrijver-constant}
Let $\mathcal{F}$ be the filtration associated to a covering
polyhedron $\mathcal{Q}$ and let $V(\mathcal{Q})$ be the vertex set of $\mathcal{Q}$. If 
$\alpha(\mathcal{Q})=\min\{|v|\colon v\in V(\mathcal{Q})\}$, 
then $\alpha(\mathcal{Q})=\widehat{\alpha}(\mathcal{F})$.
\end{corollary}

\begin{proof} The optimal value of the linear program of Theorem~\ref{schrijver-number-lp}
is equal to $|v|$ for some vertex $v$ of $\mathcal{Q}$. Thus, it
suffices to note that $|v|\leq |a|$ for any $a\in\mathcal{Q}$. This
follows from Proposition~\ref{cpr}.
\end{proof}

\begin{proposition}\label{integral-qc} 
Let $\mathcal{Q}(C)$ be a covering polyhedron and let 
$\mathcal{F}=\{I_n\}_{n=0}^\infty$ be its associated filtration. 
Then $\alpha_{\mathcal{F}}(1)\geq \widehat{\alpha}(\mathcal{F})$, with equality if
$\mathcal{Q}(C)$ is integral.
\end{proposition}

\begin{proof} By Lemma~\ref{subadditive-g}, $\widehat{\alpha}(\mathcal{F})$ is the infimum of
all $\alpha_\mathcal{F}(n)/n$ $(n\in\mathbb{N}_+)$. Thus,
$\alpha_\mathcal{F}(1)/1\geq\widehat{\alpha}(\mathcal{F})$. Now, assume that $\mathcal{Q}(C)$ is integral. 
By Corollary~\ref{schrijver-constant}, $\widehat{\alpha}(\mathcal{F})$ is equal to $|v|$ for some vertex $v$ of
$\mathcal{Q}(C)$. As $\mathcal{Q}(C)$ is integral, $v$ is integral,
and consequently $t^v\in I_1$. Thus, $|v|=\deg(t^v)\geq 
\alpha_{\mathcal{F}}(1)$, and $\widehat{\alpha}(\mathcal{F})$ is equal to $\alpha_{\mathcal{F}}(1)$.
\end{proof}

\begin{corollary}\label{jun12-21} Let $I$ be a complete
ideal of $S$ minimally generated by monomials $t^{v_1},\ldots,t^{v_q}$ of degree $d$ and let
$\alpha(\overline{I^n})$ be the initial degree of $\overline{I^n}$, then
$\lim_{n\rightarrow\infty}{\alpha(\overline{I^n})}/{n}=|v_q|$.
\end{corollary}
\begin{proof} Let $A$ be the incidence matrix of $I$, let
$u_1,\ldots,u_r$ be the vertices of $\mathcal{Q}(A)$, and let $B$ be
the matrix with column vectors $u_1,\ldots,u_r$. The filtration 
associated to $\mathcal{Q}(B)$ is $\mathcal{F}=\{\overline{I^n}\}_{n=0}^\infty$
because $\mathcal{Q}(B)$ is the Newton polyhedron of $I$ (see
Proposition~\ref{np-qa}). Since $\mathcal{Q}(B)$ is integral, 
using that $\alpha(I)=|v_q|$ and $I=\overline{I}$, by
Proposition~\ref{integral-qc}, the equality $\widehat{\alpha}(\mathcal{F})=|v_q|$ follows.  
\end{proof}

\begin{corollary}\cite[Theorem~4.3]{Seceleanu-packing}\label{jun12-21-1}
Let $I$ be a squarefree monomial ideal. If $\mathcal{Q}(I)$ is 
integral, then $\widehat{\alpha}(I)=\alpha(I)$ and
$\widehat{\alpha}(I^\vee)=\alpha(I^\vee)$.
\end{corollary}

\begin{proof} The filtration
$\mathcal{F}$ associated to $\mathcal{Q}(I^\vee)$ is
given by $I_n=I^{(n)}$ for $n\geq 1$ (Eq.~\eqref{jun21-21-1}). 
Thus, $I=I_1$
because $I$ is squarefree. Then, $\widehat{\alpha}(\mathcal{F})$ is the Waldschmidt constant
$\widehat{\alpha}(I)$ 
\cite{Cooper-symbolic,Francisco-TAMS}. According to 
\cite[Theorem~1.17]{cornu-book} $\mathcal{Q}(I)$ is integral if and
only if $\mathcal{Q}(I^\vee)$ is integral. Then, by
Proposition~\ref{integral-qc}, 
we get
$\widehat{\alpha}(I)=\widehat{\alpha}(\mathcal{F})=\alpha_\mathcal{F}(1)=\alpha(I)$
and $\widehat{\alpha}(I^\vee)=\alpha(I^\vee)$. 
\end{proof}

\begin{proposition}\label{equality-in} Let $\mathcal{Q}(C)$ be a covering polyhedron and let 
$\mathcal{F}=\{I_n\}_{n=0}^\infty$ be its associated filtration. The
following hold.
\begin{enumerate}
\item[(a)] $\overline{I_1^n}=I_n$ for all $n\geq 1$ if and only if 
$\mathcal{Q}(C)$ is integral.
\item[(b)] $I_1^n=I_n$ for all $n\geq 1$ if and only if 
$\mathcal{Q}(C)$ is integral and $I_1$ is normal.
\end{enumerate}
\end{proposition}

\begin{proof} (a): $\Rightarrow$) As the Newton polyhedron ${\rm NP}(I_1)$
of $I_1$ is integral it suffices to show the equality
${\rm NP}(I_1)=\mathcal{Q}(C)$. The inclusion ``$\subset$'' holds in
general. Indeed, let $G(I_1)$ be the minimal generating set of $I_1$. 
Note that the set $\{a\vert\, t^a\in G(I_1)\}$ is contained in 
$\mathcal{Q}(C)$ by definition of $I_1$. Hence 
$$ 
{\rm NP}(I_1)=\mathbb{R}_+^s+{\rm conv}(\{a\vert\, t^a\in
G(I_1)\})\subset\mathcal{Q}(C).
$$
\quad To show the inclusion ``$\supset$'' take any vertex $a$ of
$\mathcal{Q}(C)$. As $a$ has rational entries, there is
$n\in\mathbb{N}_+$ such that $na\in\mathcal{Q}(C)\bigcap\mathbb{N}^s$.
Then, $(na)/n\in\mathcal{Q}(C)$ and $t^{na}\in I_n$. Thus, by
hypothesis, one has $t^{na}\in\overline{I_1^n}$. Then, by
Proposition~\ref{np-qa}(a), we get that $a=(na)/n$ is in ${\rm NP}(I_1)$.

(a): $\Leftarrow$) By Lemma~\ref{filtration}, $I_n$ is complete and 
$\overline{I_1^n}\subset I_n$. Let $\beta_1,\ldots,\beta_r$ be the
vertices of $\mathcal{Q}(C)$. As $\mathcal{Q}(C)$ is integral, $\beta_i$
is integral for $i=1,\ldots,r$.  To show the inclusion
$I_n\subset\overline{I_1^n}$ take $t^a$ in $I_n$, that is, $a$ is in
$n\mathcal{Q}(C)$. By Proposition~\ref{cpr}, one has
$$
\mathcal{Q}(C)=\mathbb{R}_+^s+{\rm
conv}(\beta_1,\ldots,\beta_r).
$$
\quad Hence, using Lemma~\ref{change-of-coefficients}, we get 
$(t^a)^p\in(I_1^n)^p$ for some $p\geq 1$ and, by
Lemma~\ref{icd}, $t^a\in \overline{I_1^n}$. 

(b): $\Rightarrow$) Recall that by Lemma~\ref{filtration} one has 
$I_1^n\subset\overline{I_1^n}\subset I_n$ for all $n\geq 1$. Hence, 
$I_1^n=\overline{I_1^n}=I_n$ for all $n\geq 1$. The equality on the left 
shows that $I_1$ is normal and, by part (a), the equality on the
right shows that $\mathcal{Q}(C)$ is integral.

(b): $\Leftarrow$) This follows at once from part (a).
\end{proof}

\begin{definition}\label{mfmc-def} Let $\mathcal{C}$ be a clutter and
let $A$ be the incidence matrix of $I=I(\mathcal{C})$. The clutter
$\mathcal{C}$ has the \textit{max-flow min-cut} property if both sides 
of the LP-duality equation
\begin{equation}\label{jun6-2-03}
{\rm min}\{\langle \alpha,x\rangle \vert\, x\geq 0; xA\geq 1\}=
{\rm max}\{\langle y, 1\rangle \vert\, y\geq 0; Ay\leq\alpha\} 
\end{equation}
have integral optimum solutions $x$ and $y$ for each non-negative
integral vector $\alpha$. The clutter $\mathcal{C}$ is called 
\textit{Fulkersonian} if the covering polyhedron $\mathcal{Q}(I)$ of
$I$ is integral. 
\end{definition}

\begin{corollary}\label{ntf-char} Let
$I$ be a squarefree monomial ideal. The following hold. 
\begin{enumerate}
\item[(a)] \cite{clutters} 
$I^n=I^{(n)}$ for all $n\geq 1$ if and only if $\mathcal{Q}(I)$ is integral
and $I$ is normal. 
\item[(b)] \cite{clutters,Trung} $\overline{I^n}=I^{(n)}$
for all $n\geq 1$ if and only if $\mathcal{Q}(I)$ is integral.
\item[(c)] 
{\rm (\cite[Corollary~3.14]{clutters}, \cite[Theorem
1.4]{hhtz})} 
If $I$ is the edge ideal of
the clutter $\mathcal{C}$, then $I^n=I^{(n)}$ for all
$n\geq 1$ if and only if $\mathcal{C}$ has the max-flow min-cut
property.
\end{enumerate}
\end{corollary}

\begin{proof} The filtration of $\mathcal{Q}(I^\vee)$ is 
$\mathcal{F}=\{I^{(n)}\}_{n=0}^\infty$ (Eq.~\eqref{jun21-21-1}).
Then, by Proposition~\ref{equality-in}, $\overline{I^n}=I^{(n)}$ for all $n\geq 1$ if and only if
$\mathcal{Q}(I^\vee)$ is integral, and $I^n=I^{(n)}$ for all $n\geq
1$ if and only if $\mathcal{Q}(I^\vee)$ is integral and $I$ is normal. Therefore (a) and (b) follow by
recalling that 
$\mathcal{Q}(I)$ is integral if and
only if $\mathcal{Q}(I^\vee)$ is integral \cite[Theorem~1.17]{cornu-book}.
Part (c) follows from (a), see \cite[Theorem~3.4]{clutters}. 
\end{proof}

Let $\mathcal{Q}=\mathcal{Q}(C)$ be the covering polyhedron of an
$s\times m$ matrix $C$ and let
${\rm SC}(\mathcal{Q})$ be the Simis cone defined in
Eq.~\eqref{simis-cone-def}. By \cite[Lemma~5.4]{korte} there exists a finite set 
${\mathcal H}\subset \mathbb{N}^{s+1}$ such 
that 
\begin{equation}\label{hb-eq}
{\rm SC}({\mathcal Q})=\mathbb{R}_+{\mathcal H}\ 
\mbox{ and }\ \mathbb{Z}^{s+1}\textstyle\bigcap
\mathbb{R}_+{\mathcal H}=\mathbb{N}{\mathcal H}, 
\end{equation}
where $\mathbb{N}{\mathcal H}$ is the additive subsemigroup of 
$\mathbb{N}^{s+1}$ generated 
by ${\mathcal H}$. 
If $\mathcal{H}$ 
is minimal, with respect to inclusion, then $\mathcal{H}$ is unique
\cite{Schr1} and is called the \textit{Hilbert basis} 
of ${\rm SC}(\mathcal{Q})$.

\begin{theorem}\cite{Schr1}\label{hb-description}
The Hilbert basis of ${\rm SC}(\mathcal{Q})$ is
the set of all integral 
vectors $0\neq \alpha\in{\rm SC}(\mathcal{Q})$ such
that $\alpha$ is not the sum of two other non-zero integral vectors in
${\rm SC}(\mathcal{Q})$. 
\end{theorem}

\begin{corollary} Let $\mathcal{H}$ be the Hilbert basis of
${\rm SC}(\mathcal{Q})$. Then, the non-zero entries of any $\gamma$ in
$\mathcal{H}$ are relatively prime. 
\end{corollary}

\begin{proof} Let $k$ be the gcd of the non-zero entries of $\gamma$.
Then, $\gamma=k\gamma'$ for some $\gamma'\in\mathbb{N}^s$. Thus, 
$\gamma'=\gamma/k$ is in ${\rm SC}(\mathcal{Q})\bigcap\mathbb{N}^s$.
Hence, by Theorem~\ref{hb-description}, $k=1$.
\end{proof}

\begin{proposition}\label{scv}
Let ${\rm SC}(\mathcal{Q})$ be the Simis cone of the covering 
polyhedron $\mathcal{Q}=\mathcal{Q}(C)$ and let
$V(\mathcal{Q})=\{\beta_1,\ldots,\beta_r\}$ be the vertex set of
$\mathcal{Q}$. The following hold.
\begin{enumerate}
\item[(a)] ${\rm SC}(\mathcal{Q})=\mathbb{R}_+\{e_1,\dots,e_s,
(\beta_1,1),\ldots,(\beta_r,1)\}$.
\item[(b)]
$\mathbb{R}_+e_1,\ldots,\mathbb{R}_+e_s,\mathbb{R}_+(\beta_1,1), 
\ldots,\mathbb{R}_+(\beta_r,1)$ are 
the extreme rays of ${\rm SC}(\mathcal{Q})$.
\item[(c)] If $\mathcal{H}$ is the Hilbert basis of ${\rm
SC}(\mathcal{Q})$, then $e_i\in\mathcal{H}$ for $i=1,\ldots,s$ and for
each $(\beta_i,1)$ there is a unique $0\neq n_i\in\mathbb{N}$ such
that $n_i(\beta_i,1)\in\mathcal{H}$.
\item[(d)] If $\beta_i$ is integral, then $(\beta_i,1)\in\mathcal{H}$.
\end{enumerate}
\end{proposition}

\begin{proof} (a): Let $\mathcal{B}'$ be the set
$\{e_1,\ldots,e_s,(\beta_1,1),\ldots,(\beta_r,1)\}$ and 
let $c_1,\ldots,c_m$ be the column vectors of $C$. Given a vector
$x=(x_1,\ldots,x_{s+1})\in\mathbb{R}_+^{s+1}$. Consider the following conditions
\begin{enumerate} 
\item[(i)] $x\in{\rm SC}(\mathcal{Q})$, that is, $\langle x,(c_i,-1)\rangle\geq 0$ for all $i$.
\item[(ii)] $x\in\mathbb{R}_+\mathcal{B}'$. 
\item[(iii)] $x_{s+1}>0$ and
$x_{s+1}^{-1}(x_1,\ldots,x_s)\in\mathcal{Q}$.
\item[(iv)] $x_{s+1}>0$ and
$\langle x_{s+1}^{-1}(x_1,\ldots,x_s),c_i\rangle\geq 1$ for all $i$.
\end{enumerate}
\quad Note that all conditions are equivalent if $x_{s+1}>0$. Indeed,
(i), (iii) and (iv) are clearly equivalent and, from the equality
$\mathcal{Q}=\mathbb{R}_+^s+{\rm conv}(\beta_1,\ldots,\beta_r)$ of
Proposition~\ref{cpr}, it
follows  that (ii) and (iii) are equivalent. If $x_{s+1}=0$, then the
vector $x$
is in both ${\rm SC}(\mathcal{Q})$ and $\mathbb{R}_+\mathcal{B}'$.
Thus, conditions (i) and (ii) are equivalent, that is, ${\rm SC}(\mathcal{Q})$ is
equal to $\mathbb{R}_+\mathcal{B}'$.

(b): Next we show that $\mathcal{B}'$ is a set of representatives for the
set of all extreme rays of the Simis cone of $\mathcal{Q}$. 
As ${\rm SC}(\mathcal{Q})=\mathbb{R}_+\mathcal{B}'$, by
\cite[Proposition~1.1.23]{monalg-rev}, any extreme ray of
$\mathbb{R}_+\mathcal{B}'$ is either equal to $\mathbb{R}_+e_i$ for
some $1\leq i\leq s$ or is equal to $\mathbb{R}_+(\beta_j,1)$ for
some $1\leq j\leq r$. 
The cone generated by $e_i$, $i=1,\ldots,s$, is an extreme ray of
${\rm SC}(\mathcal{Q})$ because 
$$\mathbb{R}_+e_i=H_{e_1}\textstyle\bigcap\cdots\bigcap H_{e_{i-1}}\bigcap
H_{e_{i+1}}\bigcap\cdots\bigcap H_{e_{s+1}}\bigcap{\rm
SC}(\mathcal{Q}),
$$
where $H_{e_1},\ldots,H_{e_{s+1}}$ are support hyperplanes of ${\rm
SC}(\mathcal{Q})$ and $H_{e_{k}}\bigcap{\rm
SC}(\mathcal{Q})$ is a face of ${\rm
SC}(\mathcal{Q})$ for $1\leq k\leq s+1$. 
Let $\beta$ be a vertex of $\mathcal{Q}$. By
\cite[Corollary~1.1.47]{monalg-rev}, there are $s$ linearly independent 
vectors $e_{j_1},\ldots,e_{j_\ell},c_{j_{\ell+1}},\ldots,c_{j_s}$ such
that $\beta$ is the unique solution of the linear system 
$$
\langle y,e_{j_k} \rangle=0,\,k=1,\ldots,\ell,\,\ \langle y,c_{j_k}
\rangle=1,\, k=\ell+1,\ldots,s. 
$$
\quad Therefore, using the equivalence of (i) and (iii) when
$x_{s+1}>0$, we obtain
$$\mathbb{R}_+(\beta,1)=H_{e_{j_1}}\textstyle\bigcap\cdots\bigcap
H_{e_{j_\ell}}\bigcap
H_{(c_{j_{\ell+1}},-1)}\bigcap\cdots\bigcap H_{(c_{j_s},-1)}\bigcap{\rm
SC}(\mathcal{Q}),
$$
and consequently $\mathbb{R}_+(\beta,1)$ is an intersection of faces
of ${\rm SC}(\mathcal{Q})$. Hence, $\mathbb{R}_+(\beta,1)$ is a face
of ${\rm SC}(\mathcal{Q})$, that is,
$\mathbb{R}_+(\beta,1)$ is an extreme ray of ${\rm SC}(\mathcal{Q})$.

(c): By Theorem~\ref{hb-description}, $e_i\in\mathcal{H}$ for
$i=1,\ldots,s$. By part (b), $\mathbb{R}_+(\beta_i,1)$ is a face of dimension
$1$ of ${\rm SC}(\mathcal{Q})$. Then, using the equality ${\rm
SC}(\mathcal{Q})=\mathbb{R}_+(\mathcal{H})$ and
\cite[Proposition~1.1.23]{monalg-rev}, we obtain 
that $\mathbb{R}_+(\beta_i,1)=\mathbb{R}_+\gamma$ for some
$\gamma=(\gamma_1,\ldots,\gamma_{s+1})$ in $\mathcal{H}$. Then
\begin{align*}
&(\beta_i,1)=\lambda_i(\gamma_1,\ldots,\gamma_{s+1}),\
\lambda_i\in\mathbb{Q}_+\ \therefore\\
&\ \ \ \ \ 1=\lambda_i\gamma_{s+1} \mbox{ and
}(\beta_i,1)=\gamma_{s+1}^{-1}(\gamma_1,\ldots,\gamma_{s+1}).
\end{align*}
\quad Thus, making $n_i=\gamma_{s+1}$, we obtain
$n_i(\beta_i,1)=\gamma\in\mathcal{H}$. To show that $n_i$ is unique
assume that there is $0\neq n_i'\in\mathbb{N}$ such that
$n_i'(\beta_i,1)\in\mathcal{H}$. We may assume $n_i\geq n_i'$. Then
$$
n_i(\beta_i,1)=n_i'(\beta_i,1)+(n_i-n_i')(\beta_i,1).
$$
\quad Hence, by Theorem~\ref{hb-description}, we get $n_i=n_i'$
because $(n_i-n_i')(\beta_i,1)$ is an integral vector in ${\rm
SC}(\mathcal{Q})$.

(d): Assume that $\beta_i$ is integral and let $\gamma$ be as in the
proof of part (c). Then $\gamma=\gamma_{s+1}(\beta_i,1)$ and, by
Theorem~\ref{hb-description}, we obtain that $\gamma_{s+1}=1$ because
$(\beta_1,1)$ is an integral vector in the Simis cone ${\rm SC}(\mathcal{Q})$. Thus,
$(\beta_1,1)\in\mathcal{H}$.
\end{proof}

\begin{theorem}\label{hilbert-basis-filtration}
Let $\mathcal{R}(\mathcal{F})$ be the Rees algebra of the filtration 
$\mathcal{F}=\{I_n\}_{n=0}^{\infty}$ associated to a covering
polyhedron $\mathcal{Q}=\mathcal{Q}(C)$ and let $\mathcal{H}$ be the
Hilbert basis of ${\rm SC}(\mathcal{Q})$. The following hold. 
\begin{enumerate}
\item[(a)] If $K[\mathbb{N}{\mathcal H}]$ is the semigroup 
ring of $\mathbb{N}{\mathcal H}$, then
$\mathcal{R}(\mathcal{F})=K[\mathbb{N}{\mathcal H}]$.
\item[(b)] $\mathcal{R}(\mathcal{F})$ is a Noetherian normal finitely 
generated $K$-algebra.
\item[(c)] There exists an integer $p\geq 1$ such that $(I_p)^n=I_{np}$
for all $n\geq 1$. 
\end{enumerate}
\end{theorem}

\begin{proof} (a): Recall that $K[\mathbb{N}{\mathcal H}]= 
K[\{t^az^n\vert\, (a,n)\in\mathbb{N}{\mathcal H}\}]$. Let $c_1,\ldots,c_m$ be the column vectors of $C$.
 To show the
inclusion $\mathcal{R}(\mathcal{F})\subset K[\mathbb{N}{\mathcal H}]$ take 
$t^az^n\in I_nz^n$, that is, $t^a\in I_n$. Thus,
$a/n\in\mathcal{Q}(C)$, that is, $\langle a/n,c_i\rangle\geq 1$ for
all $i$. Hence, $\langle(a,n),(c_i-1)\rangle\geq 0$ for all $i$, that
is, $(a,n)\in {\rm SC}(\mathcal{Q})$ and, since $(a,n)$ is an integral
vector, we get $(a,n)\in\mathbb{N}\mathcal{H}$. Thus, $t^az^n\in
K[\mathbb{N}\mathcal{H}]$. To show the inclusion
$\mathcal{R}(\mathcal{F})\supset K[\mathbb{N}{\mathcal H}]$ take 
$t^az^n$ with $(a,n)\in\mathbb{N}{\mathcal H}$. Then $(a,n)$ is 
in $\mathbb{R}_+\mathcal{H}={\rm SC}({\mathcal Q})$, and consequently
$\langle (a,n),(c_i,-1)\rangle\geq 0$ for all $i$. Hence,  
$\langle a/n,c_i\rangle\geq 1$ for all $i$, that is,
$a/n\in\mathcal{Q}$. Thus, $t^az^n\in I_nz^n$, 
and $t^az^n$ is in $\mathcal{R}(\mathcal{F})$.

(b): The semigroup ring $K[\mathbb{N}\mathcal{H}]$ is generated, as a
$K$-algebra, by the finite set $F$ of all $t^az^n$ with
$(a,n)\in\mathcal{H}$. Hence $\mathcal{R}(\mathcal{F})$ is Noetherian
by the equality of part (a). Using \cite[Theorem~9.1.1]{monalg-rev},
we get that the integral closure of $K[F]$ is given by
$$
\overline{K[F]}=K[\{t^az^n\vert\,
(a,n)\in\mathbb{Z}\mathcal{H}\textstyle\bigcap\mathbb{R}_+\mathcal{H}\}].
$$
\quad Using Eq.~\eqref{hb-eq}, we obtain
$\mathbb{N}\mathcal{H}=\mathbb{Z}^{s+1}\textstyle\bigcap\mathbb{R}_+\mathcal{H}
\supset\mathbb{Z}\mathcal{H}\textstyle\bigcap\mathbb{R}_+\mathcal{H}
\supset\mathbb{N}\mathcal{H}$ and one has equality everywhere.
Therefore
$$
\overline{K[\mathbb{N}\mathcal{H}]}=\overline{K[F]}=K[\{t^az^n\vert\,
(a,n)\in\mathbb{Z}^{s+1}\textstyle\bigcap\mathbb{R}_+\mathcal{H}\}]=K[\mathbb{N}\mathcal{H}].
$$
\quad Thus, $K[\mathbb{N}\mathcal{H}]$ is normal and, by part (a),
$\mathcal{R}(\mathcal{F})$ is normal.

(c): By part (b), the Rees algebra $\mathcal{R}(\mathcal{F})$ of the filtration
$\mathcal{F}$ is Noetherian. Then, this part follows directly from 
\cite[p.~818, Proposition~2.1]{SchenzelFiltrations}.
\end{proof}

\begin{proposition}\label{jun13-21} Let $\mathcal{Q}=\mathcal{Q}(C)$ be a covering
polyhedron and let $\mathcal{F}=\{I_n\}_{n=0}^\infty$ be its
associated filtration. Then there exists an integer $k\geq 1$ such that 
$\widehat{\alpha}(\mathcal{F})={\alpha_\mathcal{F}(nk)}/{nk}$ for all $n\geq 1$.
\end{proposition}

\begin{proof} By Theorem~\ref{hilbert-basis-filtration}(c), 
there is $k\geq 1$ such that $(I_k)^n=I_{nk}$
for all $n\geq 1$. Then 
$$
\alpha_\mathcal{F}(nk)=\min\{\deg(t^a)\vert\, t^a\in I_{nk}\}=
\min\{\deg(t^a)\vert\, t^a\in (I_k)^n\}=n \alpha_\mathcal{F}(k),
$$
and $\alpha_\mathcal{F}(nk)/nk=\alpha_\mathcal{F}(k)/k$ for all $n\geq 1$.
As $\{\alpha_\mathcal{F}(nk)/nk\}_{n=1}^\infty$ is a subsequence of
$\{\alpha_\mathcal{F}(n)/n\}_{n=1}^\infty$,
taking limits in the last equality 
gives $\widehat{\alpha}(\mathcal{F})=\alpha_\mathcal{F}(nk)/nk=\alpha_\mathcal{F}(k)/k$
for all $n\geq 1$. 
\end{proof}

\section{Rees algebras of filtrations: resurgence
comparison}\label{section-comparison}

The next result was proved in \cite[Lemma~4.1]{Francisco-TAMS} for the filtration
of symbolic powers. The proof of \cite[Lemma~4.1]{Francisco-TAMS}
works for any strict Noetherian filtration. A filtration
is \textit{Noetherian} if its Rees algebra is Noetherian.

\begin{lemma}\label{lemma4.1} Let $\mathcal{F}=\{I_n\}_{n=0}^\infty$ be a strict
Noetherian filtration of ideals of
$S$, and let $\{m_n\}$ and $\{r_n\}$ be sequences of positive integers such 
that
$\lim_{n\rightarrow\infty}m_n=\lim_{n\rightarrow\infty}r_n=\infty$,
$I_{m_n}\subset I_1^{r_n}$ for all $n\geq 1$, and
$\lim_{n\rightarrow\infty}m_n/r_n=h$ for some $h\in\mathbb{R}$. Then
$\widehat{\rho}(\mathcal{F})\leq h$. 
\end{lemma}

\begin{proof} It follows from the proof
\cite[Lemma~4.1]{Francisco-TAMS}.
\end{proof} 

\begin{proposition}\label{comparison-ic} Let $\mathcal{F}=\{I_n\}_{n=0}^\infty$ be a strict
Noetherian filtration of monomial ideals of
$S$ such that $I_n$ is complete for all $n\geq 1$. 
The following hold.
\begin{enumerate}
\item[(a)] $\emptyset\neq\{m/r
\mid 
I_{mt}\not\subset \overline{I_1^{rt}} \text{ for all } t\gg
0\}\subset\left\{m/r
\mid 
I_{mt}\not\subset{I_1^{rt}} \text{ for all } t\gg
0\right\}$. 
\item[(b)] There exists $p\geq 1$ such that
$(I_p)^\ell=I_{p\ell}$ for all $\ell\geq 1$ and
$\widehat{\rho}_{ic}(\mathcal{F})\leq
\widehat{\rho}(\mathcal{F})\leq p$. 
\item[(c)]
$\widehat{\rho}_{ic}(\mathcal{F})=\widehat{\rho}(\mathcal{F})$.
\end{enumerate}
\end{proposition}
\begin{proof} (a): The
inclusion is clear because $I_1^{rt}\subset\overline{I_1^{rt}}$. 
To show that the left hand side of the inclusion is
not empty take $t\geq 1$ and pick two positive integers $m,r$ such
that $m<r$. It suffices to show that $I_{mt}\not\subset
\overline{I_1^{rt}}$. By contradiction assume that $I_{mt}\subset
\overline{I_1^{rt}}$. Then
$$
I_1^{rt}\subset I_{rt}\subset I_{mt}\subset
\overline{I_1^{rt}}.
$$
\quad Hence, taking integral closures and using that $I_n$ is
complete for all $n\geq 1$, we get $I_{rt}=I_{mt}$, a
contradiction since $\mathcal{F}$ is a strict filtration.

(b): As the filtration
$\mathcal{F}$ is Noetherian, by \cite[p.~818,
Proposition~2.1]{SchenzelFiltrations}, there is an integer $p\geq 1$ such that
$(I_p)^\ell=I_{p\ell}$ for all $\ell\geq 1$. The inequality 
$\widehat{\rho}_{ic}(\mathcal{F})\leq\widehat{\rho}(\mathcal{F})$ is
clear by part (a). To show the inequality
$\widehat{\rho}(\mathcal{F})\leq p$, let $m,r$ be positive integers
such that $I_{mt}\not\subset I_1^{rt}$ for all $t\gg 0$. It suffices
to show $m/r\leq p$. By contradiction assume that $m>rp$. Then
$mt>rpt$ and consequently
$$ 
I_{mt}\subset I_{rpt}=(I_p)^{rt}\subset I_1^{rt}\mbox { for all }t\gg 0,
$$
a contradiction. Thus, $m/r\leq p$ and
$\widehat{\rho}(\mathcal{F})\leq p$. 

(c): By part (b) one has the inequality $\widehat{\rho}(\mathcal{F})\geq
\widehat{\rho}_{ic}(\mathcal{F})$. We proceed by contradiction. 
Assume that
$\widehat{\rho}(\mathcal{F})>\widehat{\rho}_{ic}(\mathcal{F})$. Pick 
$m/r$, $m,r\in\mathbb{N}_+$, such that 
$\widehat{\rho}(\mathcal{F})>m/r>\widehat{\rho}_{ic}(\mathcal{F})$. By
the inequality on the right, there is an increasing
sequence   
$\{t_i\}_{i=1}^\infty$ such that $\lim_{i\to \infty}t_i=\infty$ and
$I_{mt_i}\subset\overline{I_1^{rt_i}}$ for all $i\in \N$. 
By \cite[Theorem~7.58]{bookthree}, there exists $k\geq 1$ such that
$\overline{I_1^n}=I_1^{n-k}\overline{I_1^k}$ for all $n\geq k$. 
Hence, $\overline{I_1^{rt_i}}=I_1^{rt_i-k}\overline{I_1^k}$ for all
$rt_i\geq k$, and thus
$$
I_{mt_i}\subset\overline{I_1^{rt_i}}\subset I_1^{rt_i-k}.
$$
\quad Now set $m_i=mt_i$ and
$r_i=rt_i-k$. We have   
\[
\lim_{i\to\infty}{m_i}/{r_i}=\lim_{i\to\infty}{mt_i}/{(rt_i-k)}={m}/{r}.
\]
\quad Therefore, by Lemma~\ref{lemma4.1}, we get
$\widehat{\rho}(\mathcal{F})\leq m/r$, a contradiction.
\end{proof}

\begin{lemma}\label{lem:a} 
If the Rees algebra $\mathcal{R}(\mathcal{F})$ of a filtration 
$\mathcal{F}=\{I_n\}_{n=0}^{\infty}$ is a finitely generated
$S$-algebra, then there exist positive integers $p$ and $k$ such that
$I_n= I_{k}(I_p)^{n/p-k/p}=I_kI_{n-k}$ for all $n\geq k$ that satisfy
$n\equiv k\!\pmod p$. Furthermore, $I_n\subset (I_p)^{\lfloor (n-k)/p\rfloor}$
for all $n\geq k$.
\end{lemma}

\begin{proof} The Rees algebra $\mathcal{R}(\mathcal{F})$ of
$\mathcal{F}$ is Noetherian because $\mathcal{R}(\mathcal{F})$ is
finitely generated as $S$-algebra and $S$ is Noetherian. Then, by \cite[p.~818,
Proposition~2.1]{SchenzelFiltrations}, there is an integer $p\geq 1$ such that
$(I_p)^\ell=I_{p\ell}$
for all $\ell\geq 1$. Let $\cR^{(p)}(\mathcal{F})$ be the $p$-th Veronese subring 
of $\cR(\mathcal{F})$. Then 
\begin{equation}\label{jun29-21}
\cR^{(p)}(\mathcal{F}):=\textstyle\bigoplus_{\ell\geq
0}I_{p\ell}z^{p\ell}=\bigoplus_{\ell\geq 0}(I_{p})^\ell z^{p\ell}.
\end{equation}
\quad The extension
$\cR^{(p)}(\mathcal{F})\subset\cR(\mathcal{F})$ is integral. 
To show this assertion take $fz^\ell\in I_\ell z^\ell$ and note that
$(fz^\ell)^p\in (I_\ell)^p z^{p\ell}\subset I_{p\ell}z^{p\ell}$. Then, 
 $\cR(\mathcal{F})$ is a finitely generated
module over $\cR^{(p)}(\mathcal{F})$. Thus, using
Eq.~\eqref{jun29-21}, it is seen that there are
$t^{b_1}z^{n_1},\ldots,t^{b_r}z^{n_r}$ in $\mathcal{R}(\mathcal{F})$ such that
\begin{equation}\label{jun29-21-1}
\cR(\mathcal{F})=\cR^{(p)}(\mathcal{F})t^{b_1}z^{n_1}+\cdots+
\cR^{(p)}(\mathcal{F})t^{b_r}z^{n_r},
\end{equation}
where $n_1\leq\cdots\leq n_r$. Therefore, using
Eq.~\eqref{jun29-21-1}, for $n\geq n_r$ one has
$$
I_nz^n=((I_p)^{\ell_1}z^{p\ell_1})(t^{b_1}z^{n_1})+\cdots+((I_p)^{\ell_r}z^{p\ell_r})(t^{b_r}z^{n_r}),
$$
where $n=\ell_1p+n_1=\cdots=\ell_rp+n_r$ and $\ell_1\geq\cdots\geq
\ell_r$. Therefore
\begin{align}\label{feb19-21}
I_n\subset(I_p)^{\ell_1}t^{b_1}+\cdots+(I_p)^{\ell_r}t^{b_r}\subset
(I_p)^{\ell_1}I_{n_1}+\cdots+(I_p)^{\ell_r}I_{n_r}\subset I_n.
\end{align}
\quad Using that $\mathcal{F}$ is a filtration and the equality
$p(\ell_i-\ell_r)+n_i=n_r$, we obtain
\begin{equation}\label{jun29-21-3}
(I_p)^{\ell_i}I_{n_i}=(I_p)^{\ell_r}((I_p)^{\ell_i-\ell_r}I_{n_i})\subset
(I_p)^{\ell_r}(I_{p(\ell_i-\ell_r)+n_i})=(I_p)^{\ell_r}I_{n_r}
 \mbox{ for all } 1\leq i<r.
\end{equation}
\quad Hence, setting $k=n_r$ and using Eqs.~\eqref{feb19-21} 
and (\ref{jun29-21-3}), we obtain 
\begin{equation}\label{jun29-21-2}
I_n=(I_p)^{\ell_r}I_{k}=(I_p)^{n/p-k/p}I_k=I_{n-k}I_k
\end{equation}
for all $n\geq k$ that satisfy $n\equiv k\!\pmod p$. 
Next we show the inclusion 
$I_n\subset (I_p)^{\lfloor (n-k)/p\rfloor}$ for all $n\geq k$.  We can
write $n-k=\lambda p+r$, $\lambda,r\in\mathbb{N}$ and $r<p$. Thus 
$(n-r)-k=\lambda p$ and $n-r\geq k$. Then, by
Eq.~(\refeq{jun29-21-2}) and the equality $I_{\lambda
p}=(I_p)^\lambda$, one has 
$$
I_n\subset I_{n-r}=I_{n-r-k}I_k\subset I_{n-r-k}=I_{\lambda
p}=(I_p)^\lambda.
$$
\quad Hence the inclusion follows by noticing that $\lambda=\lfloor
(n-k)/p\rfloor$.
\end{proof}

\begin{corollary}\label{coro-lem:a}
Let $\mathcal{R}(\mathcal{F})$ be the Rees algebra of the filtration 
$\mathcal{F}=\{I_n\}_{n=0}^{\infty}$ of a covering
polyhedron $\mathcal{Q}(C)$. Then there exist positive integers $p$ and $k$ such that
$I_n= I_{k}(I_p)^{n/p-k/p}=I_kI_{n-k}$ for all $n\geq k$ that satisfy
$n\equiv k\!\pmod p$ and $I_n\subset (I_p)^{\lfloor (n-k)/p\rfloor}$
for all $n\geq k$.
\end{corollary}

\begin{proof} It follows directly from
Theorem~\ref{hilbert-basis-filtration}(b) and Lemma~\ref{lem:a}. 
\end{proof}

\begin{proposition}
Let $\mathcal{F}=\{I_n\}_{n=0}^\infty$, $\mathcal{F}'=\{J_n\}_{n=0}^\infty$ be filtrations
of ideals of $S$. Assume $\mathcal{F}$ and $\mathcal{F}'$ have finitely generated Rees
algebras and let $p$ be an integer such that $\cR(\mathcal{F})$ is a finitely
generated module over $\cR^{(p)}(\mathcal{F})$. Consider the additional filtration
$\mathcal{F}''=\{(J_p)^n\}_{n=0}^\infty$. Then     
\[\widehat{\rho}(\mathcal{F},\mathcal{F}'')=p\widehat{\rho}(\mathcal{F},\mathcal{F}').\]
\end{proposition}
\begin{proof}
Suppose $m,r\in\N$ are such that
${mp}/{r}>\widehat{\rho}(\mathcal{F},\mathcal{F}'')$. Then $I_{mpt}\subset 
(J_p)^{rt}$ for $t\gg 0$ and since $ (J_p)^{rt}\subset J_{prt}$ we have
$I_{mpt}\subset J_{prt}$ for $t\gg 0$. This shows that      
\[ 
\widehat{\rho}(\mathcal{F},\mathcal{F}')\leq
\inf\{{m}/{r}\mid 
{mp}/{r}>\widehat{\rho}(\mathcal{F},\mathcal{F}'')\}
={\widehat{\rho}(\mathcal{F},\mathcal{F}'')}/{p}.
\]
\quad For the opposite inequality consider $m,r\in \N$ such that
${m}/{r}>\widehat{\rho}(\mathcal{F},\mathcal{F}')$. Then there is an increasing
sequence   
$\{t_i\}$ such that $\lim_{i\to \infty}t_i=\infty$ and
$I_{mt_i}\subset J_{rt_i}$ for all $i\in \N$. 
By Lemma~\ref{lem:a} we have $J_{rt_i}\subset (J_p)^{\lfloor
{(rt_i-k)}/{p} \rfloor}$ for $rt_i>k$, and thus $I_{mt_i}\subset
(J_p)^{\lfloor
{(rt_i-k)}/{p} \rfloor}$ for all $i\in \N$. Now set $m_i=mt_i$ and
$r_i=\lfloor{(rt_i-k)}/{p} \rfloor$. We have   
\[
{(rt_i-k)}/{p}  \leq r_i \leq  ({(rt_i-k)}/{p}) +1,\ \mbox{ and }
\]
\[
{mp}/{(r-({k}/{t_i}))}= {mt_ip}/{(rt_i-k)}  \geq
{m_i}/{r_i} \geq
{mt_ip}/{(rt_i-k+p)}={mp}/{(r+({(p-k)}/{t_i}))}.  
\]
\quad By the squeeze theorem it follows that
$\lim_{i\to\infty}{m_i}/{r_i}={mp}/{r}$. By \cite[Lemma
4.1]{Francisco-TAMS} (cf. Lemma~\ref{lemma4.1}) we have   
$\widehat{\rho}(\mathcal{F},\mathcal{F}'')\leq {mp}/{r}$. We have thus shown that 
\[
\widehat{\rho}(\mathcal{F},\mathcal{F}'')\leq 
p\inf\{{m}/{r}\mid
{m}/{r}>\widehat{\rho}(\mathcal{F},\mathcal{F}')\}
=p\widehat{\rho}(\mathcal{F},\mathcal{F}'),  
\]
which finishes the proof.
\end{proof}

\section{Computing the ic-resurgence with linear
programming}\label{section-computing}
\quad The main result of this section shows that the ic-resurgence
of a strict filtration associated to a covering polyhedron can be computed using 
linear programming and gives an algorithm to compute this number.

The following notation will be used throughout this section. 
Let $\mathcal{Q}(C)$ be a covering polyhedron, let $c_1,\ldots,c_m$ be
the columns of $C$, $c_i\in\mathbb{Q}_+^s$ for all $i$, let $\mathcal{F}=\{I_n\}_{n=0}^\infty$ be the
filtration associated to $\mathcal{Q}(C)$, let ${\rm NP}(I_1)$ be
the Newton polyhedron of $I_1$, let $B$ be a rational matrix with
non-negative entries and non-zero columns such that ${\rm NP}(I_1)=\mathcal{Q}(B)$, 
let $\beta_1,\ldots,\beta_k$ be the columns of $B$, and let $n_i$ be a
positive integer such that $n_i\beta_i$ is integral for
$i=1,\ldots,k$. The existence of $B$ follows from Proposition~\ref{np-qa}.

Computing $\rho_{ic}(\mathcal{F})$ is an integer linear-fractional programming
problem essentially because the Newton polyhedron and the covering polyhedron
are defined by rational systems of linear inequalities. Indeed, a monomial
$t^a$ is in $I_n\setminus\overline{I_1^r}$ if and only if
$a/n\in\mathcal{Q}(C)$ and $a/r\notin{\rm NP}(I_1)$
(Proposition~\ref{np-qa}), that is,  $t^a$ is in
$I_n\setminus\overline{I_1^r}$  if and only if 
\begin{equation}\label{mar14-20}
\langle a,c_i\rangle\geq n\mbox{ for }i=1,\ldots,m\mbox{ and }\langle
a,n_j\beta_j\rangle\leq rn_j-1\mbox{ for some }1\leq j\leq k.
\end{equation}
\quad Let $x_1,\ldots,x_s$ be variables that correspond to
the entries of $a$ and let $x_{s+1},x_{s+2}$ be two extra variables
that correspond to $n$ and $r$, respectively. Hence, by
Eq.~\eqref{mar14-20}, for each $1\leq j\leq k$ one can associate the
following integer linear-fractional 
program:
\begin{align}
&\text{maximize }\ \ h_j(x)=\frac{x_{s+1}}{x_{s+2}}&&&&\nonumber\\
&\text{subject to }\ \langle(x_1,\ldots,x_s),c_i\rangle-x_{s+1}\geq
0,\ i=1,\ldots,m,\, x_{s+1}\geq 1 &&&&\label{integer-lfp-ic-resurgence}\\
&\quad \quad\quad\quad\quad
(x_1,\ldots,x_{s},x_{s+1},x_{s+2})\in\mathbb{N}^{s+2}&&&&\nonumber \\
&\quad \quad\quad\quad\quad  n_jx_{s+2}-\langle(x_1,\ldots,x_s),n_j\beta_j\rangle
\geq 1,\, x_{s+2}\geq 1.&&&&\nonumber 
\end{align}
\quad Note that if $\tau_j$ is the
optimal value of this program, we obtain
$$
\rho_{ic}(\mathcal{F})=\sup\left.\left\{{n}/{r}\, \right|I_n\not\subset\overline{I_1^r}\right\}=
\max\{\tau_j\}_{j=1}^k.
$$

\begin{lemma}\label{lemma-non-empty}
If $\mathcal{F}$ is a strict filtration, then $I_1\not\subset\overline{I_1^r}$ for
some $r\geq 2$.  
\end{lemma}

\begin{proof} Assume that $I_1\subset\overline{I_1^r}$ for all $r\geq 2$.
By \cite[Theorem~7.58]{bookthree}, there exists $\ell\geq 1$ such that
$\overline{I_1^r}=I_1^{r-\ell}\overline{I_1^\ell}$ for all $r\geq
\ell$. Making $r-\ell=2$, we get $I_1\subset\overline{I_1^r}\subset
I_1^2$, a contradiction. 
\end{proof}

\begin{lemma}\label{bound-ic} Let $\mathcal{Q}(C)$ be a covering
polyhedron and let $\mathcal{F}$ be its associated filtration. If
$\mathcal{F}$ is strict, then 
$\widehat{\rho}(\mathcal{F})=\widehat{\rho}_{ic}(\mathcal{F})={\rho}_{ic}(\mathcal{F})$
and this is a finite number. 
\end{lemma}

\begin{proof} By Lemma~\ref{filtration},
Theorem~\ref{hilbert-basis-filtration}, 
and Proposition~\ref{comparison-ic},
$\widehat{\rho}(\mathcal{F})=\widehat{\rho}_{ic}(\mathcal{F})<\infty$.
Next we show the equality
$\widehat{\rho}_{ic}(\mathcal{F})={\rho}_{ic}(\mathcal{F})$. First we
show the inequality $\widehat{\rho}_{ic}(\mathcal{F})\leq
{\rho}_{ic}(\mathcal{F})$. Let $n/r$ be any rational number,
$n,r\in\mathbb{N}_+$, such that
$I_{n\lambda}\not\subset\overline{I_1^{r\lambda}}$ for all
$\lambda\gg 0$. Then
$n/r=n\lambda/r\lambda\leq{\rho}_{ic}(\mathcal{F})$, and 
consequently
$\widehat{\rho}_{ic}(\mathcal{F})\leq{\rho}_{ic}(\mathcal{F})$. To
show the inequality
${\rho}_{ic}(\mathcal{F})\leq\widehat{\rho}_{ic}(\mathcal{F})$, let
$n/r$ be any rational number, $n,r\in\mathbb{N}_+$, such that
$I_n\not\subset\overline{I_1^r}$. Take any integer $\lambda\geq 1$ and
pick $t^a$ in $I_n\setminus\overline{I_1^r}$. Then, $t^{{\lambda}a}$
is in $(I_n)^\lambda\subset I_{n\lambda}$. As $t^a$ is not in
$\overline{I_1^r}$, one has that $a/r$ is not in $NP(I_1)$. Then, $\langle
a/r,\beta_i\rangle<1$ for some $1\leq i\leq k$. Since 
$n/r=n\lambda/r\lambda$, we get that $t^{\lambda a}$ is not in
$\overline{I_1^{\lambda r}}$. Therefore,
$I_{n\lambda}\not\subset\overline{I_1^{r\lambda}}$ for all 
$\lambda\geq 1$. This proves that
$n/r\leq\widehat{\rho}_{ic}(\mathcal{F})$, and consequently  
${\rho}_{ic}(\mathcal{F})\leq\widehat{\rho}_{ic}(\mathcal{F})$.  
\end{proof}

The next result gives linear programs, based on
linear-fractional programming, to compute  
the ic-resurgence $\rho_{ic}(\mathcal{F})$ of $\mathcal{F}$ 
(Example~\ref{bowtie-example}, Procedure~\ref{bowtie-procedure}, Algorithm~\ref{AS-code}).

\begin{theorem}\label{lp-resurgence-formula}
Let $\mathcal{F}=\{I_n\}_{n=0}^\infty$ be the filtration of a covering
polyhedron $\mathcal{Q}(C)$. For each $1\leq j\leq k$, let $\rho_j$
be the optimal value of the 
following linear
program with variables $y_1,\ldots,y_{s+3}$. If $\mathcal{F}$ is
strict, then $\rho_{ic}(\mathcal{F})=\max\{\rho_j\}_{j=1}^k$ and $\rho_j$ is 
attained at a rational vertex of the polyhedron $\mathcal{P}_j$ of feasible points of
Eq.~\eqref{lp-ic-resurgence}. In
particular, $\rho_{ic}(\mathcal{F})$ is rational.
\begin{align}
&\text{maximize }\ \ g_j(y)=y_{s+1}&&&&\nonumber\\
&\text{subject to }\ \langle(y_1,\ldots,y_s),c_i\rangle-y_{s+1}\geq
0,\ i=1,\ldots,m,\, y_{s+1}\geq y_{s+3} &&&&\label{lp-ic-resurgence}\\
&\quad \quad\quad\quad\quad y_i\geq 0,\, i=1,\ldots,s,\, y_{s+3}\geq 0&&&&\nonumber \\
&\quad \quad\quad\quad\quad  n_jy_{s+2}-\langle(y_1,\ldots,y_s),n_j\beta_j\rangle
\geq y_{s+3},\, y_{s+2}=1.&&&&\nonumber 
\end{align}
\end{theorem}

\begin{proof} The canonical vector $e_{s+2}$ in $\mathbb{R}^{s+3}$ is
a feasible point for Eq.~\eqref{lp-ic-resurgence} for any $1\leq
j\leq k$. Let $a=(a_1,\ldots,a_s)$ be a vector in
$\mathbb{N}^s\setminus\{0\}$ and let $n,r$ be positive integers.
Recall that, by Eqs.~\eqref{mar14-20} and
\eqref{integer-lfp-ic-resurgence}, 
$t^a$ is in $I_n\setminus\overline{I_1^r}$ if and only if there
exists $1\leq j\leq k$ such that $(a_1,\ldots,a_s,n,r)$ is a 
feasible point for the following
linear-fractional program  
\begin{align}
&\text{maximize }\ \ f_j(x)=\frac{x_{s+1}}{x_{s+2}}&&&&\nonumber\\
&\text{subject to }\ \langle(x_1,\ldots,x_s),c_i\rangle-x_{s+1}\geq
0,\ i=1,\ldots,m,\, x_{s+1}\geq 1 &&&&\label{lfp-ic-resurgence}\\
&\quad \quad\quad\quad\quad x_i\geq 0,\, i=1,\ldots,s&&&&\nonumber \\
&\quad \quad\quad\quad\quad  n_jx_{s+2}-\langle(x_1,\ldots,x_s),n_j\beta_j\rangle
\geq 1,\, x_{s+2}\geq 1&&&&\nonumber 
\end{align}
with variables $x_1,\ldots,x_{s},x_{s+1},x_{s+2}$. 
The polyhedron $\mathcal{Q}_j$ defined by the constraints of 
Eq.~\eqref{lfp-ic-resurgence} is not empty for any $1\leq j\leq k$.
Indeed, since the $c_i$'s are non-zero vectors in $\mathbb{Q}_+^s$ and
$n_j\geq 1$ it follows that $(b_1,\ldots,b_{s+2})$ is in
$\mathcal{Q}_j$ by choosing $b_{s+1}=1$ and $b_1,\ldots,b_s,b_{s+2}$ large enough
integers  (cf. Lemma~\ref{lemma-non-empty}).
We set $\rho_j'=\sup\{f_j(x)\vert\, x\in\mathcal{Q}_j\}$. 
Next we show that $\rho_j'$ is finite. Take any rational feasible
point $x$ for 
Eq.~\eqref{lfp-ic-resurgence} and pick a positive integer $\lambda$ such
that $\lambda x\in\mathbb{N}^{s+2}$. As $\lambda x$ is also feasible
for Eq.~\eqref{lfp-ic-resurgence}, that is, $\lambda x\in\mathcal{Q}_j$,
one has $I_{\lambda x_{s+1}}\not\subset\overline{I_1^{\lambda
x_{s+2}}}$. Therefore 
$$f_j(x)=\frac{x_{s+1}}{x_{s+2}}=f_j(\lambda x)
=\frac{\lambda x_{s+1}}{\lambda x_{s+2}}\leq\rho_{ic}(\mathcal{F}),
$$
and consequently, by Lemma~\ref{bound-ic}, $\rho_j'\leq
\rho_{ic}(\mathcal{F})<\infty$. This proves
$\max\{\rho_j'\}_{j=1}^k\leq \rho_{ic}(\mathcal{F})$. Below we show the
reverse inequality. First we prove that $\rho_j=\rho_j'$ for all $j$. 

As we now explain, the linear-fractional program of
Eq.~\eqref{lfp-ic-resurgence} is equivalent to the 
linear program of Eq.~\eqref{lp-ic-resurgence} \cite[Section~4.3.2,
p.~151]{boyd}. 
To show the equivalence, we first note that if $x$ is feasible for 
Eq.~\eqref{lfp-ic-resurgence} then the
point
$$
y=\left(\frac{x_1}{x_{s+2}},\ldots,\frac{x_{s}}{x_{s+2}},\frac{x_{s+1}}{x_{s+2}},\frac{x_{s+2}}{x_{s+2}},\frac{1}{x_{s+2}}\right)
$$
is feasible for Eq.~\eqref{lp-ic-resurgence}, with the same objective
value, that is, $f_j(x)=g_j(y)$. It follows that 
the optimal value $\rho_j$ of Eq.~\eqref{lp-ic-resurgence} is greater than or
equal to the optimal value $\rho_j'$ of Eq.~\eqref{lfp-ic-resurgence}, that is,
$\rho_j\geq\rho_j'$. To prove that $\rho_j=\rho_j'$, suppose to the
contrary that $\rho_j>\rho_j'$. Pick $y$ feasible for 
Eq.~\eqref{lp-ic-resurgence} such that $g_j(y)>\rho_j'$. If
$0<y_{s+3}\leq 1$, then the point 
$$
x=\left(\frac{y_1}{y_{s+3}},\ldots,\frac{y_{s}}{y_{s+3}},\frac{y_{s+1}}{y_{s+3}},\frac{y_{s+2}}{y_{s+3}}\right)
$$
is feasible for Eq.~\eqref{lfp-ic-resurgence}, with the same objective value, that
is, $g_j(y)=f_j(x)$. Hence $f_j(x)>\rho_j'$, a contradiction. If $y_{s+3}\geq 1$, then the point 
$$
x=\left({y_1},\ldots,{y_{s}},{y_{s+1}},{y_{s+2}}\right)
$$
is feasible for Eq.~\eqref{lfp-ic-resurgence} and $g_j(y)=f_j(x)$.
Thus $f_j(x)>\rho_j'$, a contradiction. If
$y_{s+3}=0$, we choose $x$ feasible
for Eq.~\eqref{lfp-ic-resurgence}. Then $x+\lambda(y_1,\ldots,y_{s+2})$ is feasible 
for Eq.~\eqref{lfp-ic-resurgence} for all $\lambda\geq 0$ and  
$$\lim_{\lambda\rightarrow\infty}f_j(x+\lambda(y_1,\ldots,y_{s+2})) = 
\lim_{\lambda\rightarrow\infty}\frac{x_{s+1}+\lambda
y_{s+1}}{x_{s+2}+\lambda
y_{s+2}}=\frac{y_{s+1}}{y_{s+2}}=y_{s+1}=g_j(y) ,$$ 
so we can find
feasible points for Eq.~\eqref{lfp-ic-resurgence} with objective values arbitrarily close to the objective value
$g_j(y)$ of $y$. Thus $f_j(x+\lambda(y_1,\ldots,y_{s+2}))>\rho_j'$ for
some $\lambda\gg 0$, a contradiction. This proves that $\rho_j=\rho_j'$. 
The integral feasible points for Eq.~\eqref{integer-lfp-ic-resurgence} and
Eq.~\eqref{lfp-ic-resurgence} are the same. Hence
$$
\rho_{ic}(\mathcal{F})=\sup\left.\left\{\frac{n}{r}\right|I_n\not\subset\overline{I_1^r}\right\}=
\sup\left.\left\{\frac{n}{r}\right|(a,n,r)\in\mathcal{Q}_j\textstyle\bigcap\mathbb{N}^{s+2};\,
\mbox{ for some }j\mbox{ and }a\right\},
$$
and consequently $\rho_{ic}(\mathcal{F})\leq\max\{\rho_j'\}_{j=1}^k$.
Thus $\rho_{ic}(\mathcal{F})=\max\{\rho_j'\}_{j=1}^k=\max\{\rho_j\}_{j=1}^k$.
Finally, for $1\leq j\leq k$, let $\mathcal{P}_j$ be the rational polyhedron of
feasible points for Eq.~\eqref{lp-ic-resurgence}. All vertices of
$\mathcal{P}_j$ are rational \cite[Proposition~1.1.46]{monalg-rev} and
the optimal value of the linear program of
Eq.~\eqref{lp-ic-resurgence} is 
attained at a vertex of $\mathcal{P}_j$ \cite[Proposition~1.1.41]{monalg-rev}.
\end{proof}

\section{The ic-resurgence of ideals of covers of edge
ideals}\label{section-resurgence-dual}

Let $G$ be a graph with vertex set $V(G)=\{t_1,\ldots,t_s\}$ and edge
set $E(G)$. A {\it coloring\/} of the vertices of $G$ is an assignment
of colors to the vertices of $G$ in such a way that adjacent vertices
have distinct colors. The {\it chromatic
number\/} of a graph $G$, denoted by
$\chi(G)$,
is the minimum number of colors in a 
coloring of $G$. Given $A\subset V(G)$, 
the \textit{induced subgraph} on $A$, denoted $G[A]$, is the maximal
subgraph of $G$ 
with vertex set $A$. A subgraph of the form $G[A]$ is called
an \textit{induced subgraph}. 
A {\it clique\/} of $G$
is a set of vertices inducing a complete subgraph. 
We also call a complete subgraph $\mathcal{K}_r$ of $G$ 
a clique. The \textit{clique
number} of $G$, denoted by $\omega(G)$, is the number of vertices in
a maximum clique in $G$. In general one has $ \omega(G)\leq\chi(G)$. 
A graph $G$ is called {\it perfect\/} if $\omega(H)=\chi(H)$
for every induced subgraph $H$ of $G$ \cite[p.~111]{diestel}. 
This notion was introduced by Berge \cite[Chapter
16]{berge-graphs-hypergraphs}.  

\begin{proposition}\cite[Proposition~2.2,
Theorem~2.10]{perfect}\label{myperfect-char} Let $J=I_c(G)$ be the ideal 
of covers of a graph $G$ and let ${\rm RC}(J)$ be the Rees cone of
$J$ defined in Eq.~\eqref{rees-cone-eq}. 
Then 
\begin{equation}\label{mundial06}
\mathrm{RC}(J)\subset\left\{(a_i)\in\mathbb{R}^{s+1}\vert\,
\textstyle\sum_{t_i\in\mathcal{K}_r}a_i
\geq (r-1)a_{s+1};\ \forall\,  {\mathcal K}_r\subset
G\right\}
\end{equation}
with equality if and only if $G$ is perfect. If $G$ is perfect, then $I_c(G)$ is normal. 
\end{proposition}

\begin{theorem}\label{jun14-21}
Let $G$ be a graph, let $I_c(G)$ be the ideal of covers of $G$, and
let $\omega(G)$ be the clique number. Then the resurgence and ic-resurgence of $I_c(G)$
satisfy
$$
\rho(I_c(G))\geq\rho_{ic}(I_c(G))\geq{2(\omega(G)-1)}/{\omega(G)}
$$
with equality everywhere if $G$ is perfect. 
\end{theorem}
\begin{proof} Setting $J=I_c(G)$, one clearly has $\rho(J)\geq\rho_{ic}(J)$
and, by Proposition~\ref{myperfect-char}, equality holds if $G$ is
perfect because in this case $J$ is normal. Next we show the second inequality. 
Let $\lambda$ be any positive integer. We set
$\omega=\omega(G)$, 
$a_\lambda=\sum_{i=1}^s\lambda e_i$, and
$b_\lambda=\lceil(\lambda \omega+1)/(\omega-1)\rceil$. We claim that
$t^{a_\lambda}\in
J^{(2\lambda)}\setminus\overline{J^{b_\lambda}}$. Recall that $J$ is
the intersection of all ideals $(t_i,t_j)$ such that $\{t_i,t_j\}\in
E(G)$ and $t_1\cdots t_s$ is in $(t_i,t_j)^2$ for all $\{t_i,t_j\}\in
E(G)$. Thus, by Lemma~\ref{anoth-one-char-spow-general}, $t_1\cdots t_s$ is in $J^{(2)}$, and
consequently $t^{a_\lambda}$ is in $J^{(2\lambda)}$. According to 
\cite[Theorem~9.1.1 and p.~509]{monalg-rev}, the
integral closure of the Rees algebra ${\mathcal{R}(J)}$ of the ideal $J$ is given by 
\begin{equation}\label{apr13-21}
\overline{\mathcal{R}(J)}=S\textstyle\bigoplus
\overline{J}z\bigoplus\cdots\bigoplus \overline{J^n}z^n\textstyle\bigoplus\cdots=
K[\{t^az^b\vert\, (a,b)\in{\rm RC}(J)\textstyle\bigcap\mathbb{Z}^{s+1}\}],
\end{equation}
where the Rees cone ${\rm RC}(J)$ of $J$ is defined in Eq.~\eqref{rees-cone-eq} of
Procedure~\ref{bowtie-procedure}.  
Pick a complete subgraph $\mathcal{K}_\omega$ of $G$ with $\omega$
vertices and set
$a_{\lambda,i}=\lambda$ for $i=1,\ldots,s$. Then one has
\begin{equation*}
1+\sum_{t_i\in\mathcal{K}_\omega}a_{\lambda,i}=1+\lambda
\omega=\frac{(\omega-1)}{1}\frac{(1+\lambda \omega)}{(\omega-1)}\leq
(\omega-1)\left\lceil\frac{1+\lambda \omega}{\omega-1}\right\rceil=(\omega-1)b_{\lambda},
\end{equation*}
and consequently $\sum_{t_i\in\mathcal{K}_\omega}a_{\lambda,i}\leq
(\omega-1)b_{\lambda}-1$. If $t^{a_\lambda}$ is in
$\overline{J^{b_\lambda}}$, then
$t^{a_\lambda}z^{b_\lambda}\in\overline{\mathcal{R}(J)}$ and, by
Eq.~\eqref{apr13-21}, we get $(a_\lambda,b_\lambda)\in{\rm RC}(J)$.
Hence, by Proposition~\ref{myperfect-char}, we get 
$$
\sum_{t_i\in\mathcal{K}_\omega}a_{\lambda,i}
\geq (\omega-1)b_\lambda,
$$
a contradiction. Thus, $t^{a_\lambda}\notin \overline{J^{b_\lambda}}$
and the claim has been proven. Therefore, $\rho_{ic}(J)\geq
2\lambda/b_\lambda$ for all $\lambda\in\mathbb{N}_+$. Hence, noticing 
that $b_\lambda\leq ((\lambda\omega+1)/(\omega-1))+1$, we obtain
$$
\rho_{ic}(J)\geq\frac{2\lambda}{b_\lambda}\geq\frac{2\lambda(\omega-1)}{\lambda\omega+\omega}
\ \ \therefore\ \  
\rho_{ic}(J)\geq \lim_{\lambda\rightarrow\infty}\frac{2\lambda(\omega-1)}{\lambda\omega+\omega}=
\frac{2(\omega-1)}{\omega}.
$$
\quad Assume that $G$ is perfect. Let $a_{s+1}, a_{s+2}$ be positive
integers such that $J^{(a_{s+1})}\not\subset \overline{J^{a_{s+2}}}$
for some $a_{s+1},a_{s+2}\in\mathbb{N}_+$. As $J$ is normal, there
exists $t^a$ in $J^{(a_{s+1})}\setminus J^{a_{s+2}}$,
$a=(a_1,\ldots,a_s)$. Then, $t^az^{a_{s+2}}$ is not in 
$\mathcal{R}(J)=\overline{\mathcal{R}(J)}$, that is, $(a,a_{s+2})$ is
not in ${\rm RC}(J)$. Hence, by Proposition~\ref{myperfect-char},
there is $2\leq r\leq s$ such that 
\begin{equation}\label{apr14-21}
\sum_{t_i\in\mathcal{K}_r}a_i
\leq (r-1)a_{s+2}-1.
\end{equation}
\quad Let $A$ be the incidence matrix of $I(G)$. 
Then, by Eq.~\eqref{jun21-21-1} of Section~\ref{section-intro}, one has  
$$J^{(a_{s+1})}=(\{t^c\vert\,
c/a_{s+1}\in\mathcal{Q}(A)\})\ \ \therefore\ \ a_i+a_j\geq a_{s+1}\
\forall\, \{t_i,t_j\}\in E(G).
$$
\quad We may assume that the vertices of $\mathcal{K}_r$ are
$t_1,\ldots,t_r$. Therefore
\begin{align}\label{apr14-21-1}
\sum_{t_i\in\mathcal{K}_r}a_i&=\frac{2(a_1+\cdots+a_r)}{2}=
\frac{(a_1+a_2)+\cdots+(a_{r-1}+a_r)+(a_r+a_1)}{2}\geq
\frac{ra_{s+1}}{2}.
\end{align}
\quad As $r\leq\omega$, using Eqs.~\eqref{apr14-21} and \eqref{apr14-21-1},
we get 
$$
\frac{a_{s+1}}{a_{s+2}}\leq
\frac{2(r-1)}{r}-\frac{2}{ra_{s+2}}\leq\frac{2(r-1)}{r}\leq\frac{2(\omega-1)}{\omega}.
$$
\quad Therefore $\rho_{ic}(J)\leq{2(\omega-1)}/{\omega}$ and the proof
is complete. 
\end{proof}

Let $G$ be a graph. A set of vertices $D$ of $G$ is called a \textit{vertex
cover} if every edge of $G$ contains at least one vertex of $D$. A 
\textit{minimal vertex cover} of $G$ is a vertex cover which
is minimal with respect to inclusion.
The number of vertices
in any smallest vertex cover of $G$, denoted by $\alpha_0(G)$, 
is called the \textit{covering number} of $G$. The height of $I(G)$ is
equal to $\alpha_0(G)$.

\begin{proposition}\label{jun14-21-1}
Let $G$ be a graph and let $I(G)$ be its edge ideal. If $H$ is an induced
subgraph of $G$ with covering number $\alpha_0(H)$ and
$\rho_{ic}(I(G))$ is the ic-resurgence of $I(G)$, then  
$$
\rho_{ic}(I(G))\geq{2\alpha_0(H)}/{|V(H)|}.
$$
\end{proposition}
\begin{proof} We may assume that the vertices of $H$ are
$t_1,\ldots,t_n$. Let $\lambda$ be any positive integer. We set 
$I=I(G)$, $\alpha_0=\alpha_0(H)$,  
$a_\lambda=\sum_{i=1}^n\lambda e_i$, and
$b_\lambda=\lceil(\lambda n+1)/2\rceil$. We claim that
$t^{a_\lambda}\in
I^{(\lambda\alpha_0)}\setminus\overline{I^{b_\lambda}}$. 
Recall that $I$ is
the intersection of all ideals $(C)$ such that $C$ is a minimal vertex
cover of $G$. Take a minimal vertex cover $C$ of $G$. Then there is
 a minimal vertex cover $C_H$ of $H$ contained in $C$. Setting 
$a=\sum_{i=1}^ne_i$, $a_i=1$ for $i=1,\ldots,n$, and $a_i=0$ for
$i>n$, one has
$$
|C\textstyle\bigcap V(H)|=\sum_{t_i\in C}a_i\geq \sum_{t_i\in
C_H}a_i=|C_H|\geq \alpha_0(H).
$$
\quad Then $t^a\in (C)^{\alpha_0}$. Thus, by
Lemma~\ref{anoth-one-char-spow-general}, 
$t^a\in I^{(\alpha_0)}$ and
consequently $t^{\lambda a}=t^{a_\lambda}\in I^{(\lambda \alpha_0)}$. The
integral closure of the Rees algebra of $I$ is given by 
\begin{equation}\label{apr16-21}
\overline{\mathcal{R}(I)}=S\textstyle\bigoplus
\overline{I}z\bigoplus\cdots\bigoplus \overline{I^n}z^n\textstyle\bigoplus\cdots=
K[\{t^az^b\vert\, (a,b)\in{\rm RC}(I)\bigcap\mathbb{Z}^{s+1}\}],
\end{equation}
see \cite[Theorem~9.1.1 and p.~509]{monalg-rev}. Using the definition
of ${\rm RC}(I)$ in Eq.~\eqref{rees-cone-eq} of Appendix~\ref{Appendix} 
and noticing that $I$ is
generated by squarefree monomials of degree $2$, we obtain the inclusion
\begin{equation}\label{helvex}
\mathrm{RC}(I)\subset\left\{(b_i)\in\mathbb{R}_+^{s+1}\vert\,
\textstyle\sum_{i=1}^sb_i
\geq 2b_{s+1}\right\}.
\end{equation}
\quad Next we show that $t^{a_\lambda}\notin\overline{I^{b_\lambda}}$.
One has the inequality 
$$
\lambda n =2\left(\frac{\lambda n+1}{2}\right)-1 
\leq 2\left\lceil\frac{\lambda n+1}{2}\right\rceil-1=2b_\lambda-1,
$$
that is, $\lambda n\leq 2b_\lambda-1$. If $t^{a_\lambda}$ is in
$\overline{I^{b_\lambda}}$, then
$t^{a_\lambda}z^{b_\lambda}\in\overline{\mathcal{R}(I)}$ and, by 
Eq.~\eqref{apr16-21}, we obtain $(a_\lambda,b_\lambda)\in{\rm
RC}(I)$. Hence, from Eq.~\eqref{helvex},  
we get $n\lambda \geq 2b_\lambda$, a contradiction. Thus,
$t^{a_\lambda}\notin \overline{I^{b_\lambda}}$
and the claim has been proven. Therefore, $\rho_{ic}(I)\geq
\lambda\alpha_0/b_\lambda$ for all $\lambda\in\mathbb{N}_+$. Hence, noticing 
that $b_\lambda\leq ((\lambda n+1)/2)+1$, we obtain
$$
\rho_{ic}(I)\geq\frac{\lambda\alpha_0}{b_\lambda}\geq\frac{2\lambda\alpha_0}{\lambda
n+3}
\ \ \therefore\ \  
\rho_{ic}(I)\geq \lim_{\lambda\rightarrow\infty}\frac{2\lambda\alpha_0}{\lambda
n+3}=
\frac{2\alpha_0}{n}.
$$
\quad Therefore $\rho_{ic}(I)\geq{2\alpha_0(H)}/|V(H)|$ and the proof
is complete. 
\end{proof}

\begin{lemma}\label{jun9-21}
Let $G$ be a graph. If $C_k$ is an induced odd cycle of length $k\geq
3$, then any minimal vertex cover $C$ of $C_k$ contains an edge of
$G$.
\end{lemma}

\begin{proof} Suppose to the contrary that $C$ is a stable set of $G$.
Setting $C'=V(C_k)\setminus C$, note that any edge of $C_k$ 
intersects $C$ and $C'$. Thus $C_k$ is a bipartite graph, a contradiction.
\end{proof}

\begin{proposition}\label{jun14-21-2} Let $G$ be a non-bipartite graph. Then 
$I_c(G)^{(n)}\subset I(G)^{(n)}$ for all $n\geq 1$ and 
$\widehat{\alpha}(I(G))\leq\widehat{\alpha}(I_c(G))$.
\end{proposition}

\begin{proof} By \cite[Lemma~3.10]{Francisco-TAMS} we need only show
$I_c(G)\subset I(G)$. Take a minimal vertex cover $D$ of $G$ and pick
an induced odd cycle $C_k$ of $G$ of length $k$. Note that $V(C_k)\bigcap D$ contains a minimal vertex cover
$C$ of $C_k$. Hence, by Lemma~\ref{jun9-21}, $C$ contains an edge $e$ of
$G$. Thus $D$ contains the edge $e$, and consequently $\prod_{t_i\in D}t_i\in I(G)$.
\end{proof}

\section{Covering polyhedra and irreducible
decompositions}\label{section-covering}

In this section we relate the
covering polyhedron and the irreducible decomposition of a monomial
ideal and study when the Newton polyhedron is the irreducible
polyhedron. 

\begin{theorem}\label{irreducible-deco-qA} 
Let $I$
be a monomial ideal of $S$,
let
$I=\bigcap_{i=1}^m\mathfrak{q}_i$ be its
irreducible decomposition, and let $\mathcal{Q}(I)$ be the covering 
polyhedron of $I$. The following hold.
\begin{enumerate}
\item[(a)] If $\mathfrak{q}_k=(t_1^{b_1},\ldots,t_r^{b_r})$,
$b_\ell\geq 1$ for all $\ell$, and
${\rm rad}(\mathfrak{q}_j)\not\subset{\rm rad}(\mathfrak{q}_k)$ for $j\neq
k$, then the vector $b^{-1}:=\sum_{i=1}^rb_i^{-1}e_i$ is a
vertex of $\mathcal{Q}(I)$.
\item[(b)] If $I$ has no embedded associated primes and 
${\rm rad}(\mathfrak{q}_j)\neq{\rm rad}(\mathfrak{q}_i)$ for $j\neq
i$, then there are $\alpha_1,\ldots,\alpha_m$ in
$\mathbb{N}^s\setminus\{0\}$ such
that $\mathfrak{q}_i=\mathfrak{q}_{\alpha_i}$ and $\alpha_i^{-1}$ is a
vertex of $\mathcal{Q}(I)$ for $i=1,\ldots,m$.
\end{enumerate}
\end{theorem}

\begin{proof} (a): Let $G(I)=\{t^{v_1},\ldots,t^{v_q}\}$ be the
minimal generating set of 
$I$ and let $A$ be the incidence matrix of $I$. Recall that
$\mathcal{Q}(I)=\mathcal{Q}(A)$. For each $j\neq k$ there is $t_{p_j}\in{\rm
rad}(\mathfrak{q}_j)$ such that $t_{p_j}\notin{\rm
rad}(\mathfrak{q}_k)=(t_1,\ldots,t_r)$. Then, $t_{p_j}^{c_{p_j}}$ is
in $G(\mathfrak{q}_j)$ for some $c_{p_j}\geq 1$. Taking the least
common multiple of all $t_{p_j}^{c_{p_j}}$, $j\neq k$, it is not hard
to show that 
there is a minimal generator $t^a$ of $\bigcap_{j\neq
k}\mathfrak{q}_j$ such that ${\rm
supp}(t^a)\subset\{t_i\}_{i=r+1}^s$. Then $t^at_\ell^{b_\ell}$ is
in $I$ for $\ell=1,\ldots,r$ and we can write
$t^at_\ell^{b_\ell}=t^{\delta_\ell}t^{v_{n_\ell}}$ for some
$t^{\delta_\ell}\in S$ and $t^{v_{n_\ell}}\in G(I)$. As
$t^{v_{n_\ell}}\in\mathfrak{q}_k$, it follows that 
$t_\ell^{b_\ell}t^{c_\ell}=t^{v_{n_\ell}}$ for some $t^{c_\ell}$ whose
support is contained in $\{t_i\}_{i=r+1}^s$.
Hence for each $1\leq 
\ell\leq r$ there is $v_{n_\ell}$ in $\{v_1,\ldots,v_q\}$ such that
$t^{v_{n_\ell}}=t_\ell^{b_\ell}t^{c_\ell}$ and  
${\rm supp}(t^{c_\ell})\subset\{t_{r+1},\ldots,t_s\}$. The vector
$b^{-1}=\sum_{i=1}^rb_i^{-1}e_i$ is in $\mathcal{Q}(A)$ because $t^{v_i}\in \mathfrak{q}_k$ for
$i=1,\ldots,q$, and since 
$\{e_i\}_{i=r+1}^s\bigcup\{v_{n_1},\ldots,v_{n_r}\}$ is linearly
independent, and  
$$
\langle b^{-1},e_i\rangle=0\ \ \ \ (i=r+1,\ldots,s); 
\ \ \ \langle b^{-1},v_{n_\ell}\rangle=1\ \ \ \ (\ell=1,\ldots,r),
$$
we get that the vector $b^{-1}$ is a basic feasible solution of the linear system
$y\geq 0$; $yA\geq 1$. Therefore, by
\cite[Theorem~2.3]{bertsimas}, $b^{-1}$ is a 
vertex of $\mathcal{Q}(A)=\mathcal{Q}(I)$. 

(b): As $I$ has no embedded primes and the irreducible decomposition
of $I$ is minimal, by permuting variables, we can apply part (a) to
$\mathfrak{q}_i$ for $i=1,\ldots,m$. Then, there are $\alpha_1,\ldots,\alpha_m$ in
$\mathbb{N}^s\setminus\{0\}$ such 
that $\mathfrak{q}_i=\mathfrak{q}_{\alpha_i}$ and $\alpha_i^{-1}$ is a
vertex of $\mathcal{Q}(A)$ for $i=1,\ldots,m$.
\end{proof}

\begin{remark}\label{jun26-21} Let $I$ be the edge ideal of a weighted oriented graph
\cite{cm-oriented-trees,WOG}, let $\mathfrak{p}=(t_1,\ldots,t_r)$ be a
minimal prime of $I$, and let $\mathfrak{q}=SI_\mathfrak{p}\bigcap S$ be
the $\mathfrak{p}$-primary component of $I$. The irreducible
decomposition of $I$ is known to be minimal \cite[Theorem~25]{WOG}.
Then $\mathfrak{q}=(t_1^{b_1},\ldots,t_r^{b_r})$, $b_i\geq 1$ for all
$i$, and by Theorem~\ref{irreducible-deco-qA}(a),
$b^{-1}=\sum_{i=1}^rb_i^{-1}e_i$ is a vertex of $\mathcal{Q}(I)$
(cf. Example~\ref{non-normal}).
\end{remark}

\begin{lemma}\label{irreducible-ideals}
Let $\mathfrak{q}=(t_1^{b_1},\ldots,t_r^{b_r})$ be an irreducible
monomial ideal of $S$, $b_i\geq 1$ for all $i$, and let $t^a$ be a
monomial of $S$. The following hold.
\begin{enumerate}
\item[(a)] ${\rm NP}(\mathfrak{q})={\rm IP}(\mathfrak{q})$ and the
set of vertices of ${\rm NP}(\mathfrak{q})$ is $V=\{b_1e_1,\ldots,b_re_r\}$.
\item[(b)] $t^a\in\overline{\mathfrak{q}^n}$ if and only if
$\langle a/n,b^{-1}\rangle\geq 1$, where 
$b^{-1}=\sum_{i=1}^rb_i^{-1}e_i$.
\item[(c)] If $\mathfrak{q}$ is normal, then $t^a\in\mathfrak{q}^{(n)}$
if and only if $\langle a/n,b^{-1}\rangle\geq 1$.
\end{enumerate} 
\end{lemma}

\begin{proof} (a): Let $A$ be the incidence matrix of
$\mathfrak{q}$, that is, $A$ is the matrix with column vectors
$b_1e_1,\ldots,b_re_r$. It is seen that the only vertex of $\mathcal{Q}(A)$ 
is $b^{-1}=\sum_{i=1}^rb_i^{-1}e_i$. Let $B$ be the $s\times 1$
matrix whose only column vector is $b^{-1}$. By
Proposition~\ref{np-qa}, ${\rm NP}(\mathfrak{q})=\mathcal{Q}(B)$.
Therefore 
$$
{\rm IP}(\mathfrak{q}):=\{x\mid x\geq 0;\ \langle x, 
b^{-1}\rangle\geq 1\}=\mathcal{Q}(B)={\rm NP}(\mathfrak{q})=\mathbb{R}_+^s+{\rm 
conv}(b_1e_1,\ldots,b_re_r).
$$
\quad To complete the proof note that the vertices of ${\rm
NP}(\mathfrak{q})$ are $b_1e_1,\ldots,b_re_r$.

(b): This part follows from (a) and Proposition~\ref{np-qa}.

(c): As $\mathfrak{q}^{(n)}=\mathfrak{q}^n=\overline{\mathfrak{q}^n}$,
by part (b), $t^a\in\mathfrak{q}^{(n)}$
if and only if $\langle a/n,b^{-1}\rangle\geq 1$. 
\end{proof}

\begin{proposition}\label{jun15-21} Let $\mathfrak{q}$ be a primary monomial ideal of
$S$. Then ${\rm NP}(\mathfrak{q})={\rm IP}(\mathfrak{q})$ if and only
if up to permutation of variables
$\mathfrak{q}=(t_1^{v_1},\ldots,t_r^{v_r})$ with $v_i\geq 1$ for all
$i$.
\end{proposition}

\begin{proof} $\Rightarrow$): By Proposition~\ref{primary-monomial}
we may assume that ${\rm rad}(\mathfrak{q})=(t_1,\ldots,t_s)$ and
also that $\mathfrak{q}$ is minimally generated by
$G(\mathfrak{q})=\{t_1^{v_1},\ldots,t_s^{v_s},t^{v_{s+1}},\ldots,t^{v_q}\}$, 
where $v_i\in\mathbb{N}_+$ for $i=1,\ldots,s$ and
$v_i\in\mathbb{N}^s\setminus\{0\}$ for
$i>s$. We proceed by contradiction assuming that $q>s$. The Newton
polyhedron of $\mathfrak{q}$ is given by
$$
{\rm NP}(\mathfrak{q})=\mathbb{R}_+^s+{\rm
conv}(v_1e_1,\ldots,v_se_s,v_{s+1},\ldots,v_q).
$$
\quad Recall that, by
Eq.~\eqref{jun4-21} of Section~\ref{section-intro}, the irreducible
decomposition of $\mathfrak{q}$ has the form 
$$
\mathfrak{q}=\mathfrak{q}_{\alpha_1}\textstyle\bigcap\cdots\bigcap\mathfrak{q}_{\alpha_m}
$$
for some $\alpha_1,\ldots,\alpha_m$ in $\mathbb{N}^s\setminus\{0\}$.
Note that all entries of each $\alpha_i$ are positive.  
For $i=1,\ldots,m$, let $H_i^+$ and $H_i$ be the following closed
halfspace and its corresponding bounding hyperplane
$$
H_i^+:=\{x\mid \langle x,\alpha_i^{-1}
\rangle\geq 1\}\ \mbox{ and }\ H_i:=\{x\mid \langle x,\alpha_i^{-1}
\rangle=1\}. 
$$
\quad Then ${\rm IP}(\mathfrak{q})$ is equal to
$(\bigcap_{i=1}^mH_i^+)\bigcap(\bigcap_{i=1}^sH_{e_i}^+)$. By removing redundant closed halfspaces in the
intersection we may assume that $(\bigcap_{i=1}^{\ell}H_i^+)\bigcap(\bigcap_{i=1}^sH_{e_i}^+)$ is an
irreducible representation of ${\rm IP}(\mathfrak{q})$
for some $1\leq\ell\leq m$ \cite[p.~12]{monalg-rev}.

\quad Case (I): The vertices of ${\rm NP}(\mathfrak{q})$ are contained
in $\{v_1e_1,\ldots,v_se_s\}$. In this case one has
\begin{align}\label{jun4-21-1} 
{\rm NP}(\mathfrak{q})&=\mathbb{R}_+^s+{\rm
conv}(v_1e_1,\ldots,v_se_s)=\{x\mid x\geq 0; \langle x,v^{-1}
\rangle\geq 1\}\\
&=\{x\mid \langle x,v^{-1}
\rangle\geq 1\}\textstyle\bigcap H_{e_1}^+\bigcap\cdots\bigcap
H_{e_s}^+,\nonumber
\end{align}
where $v^{-1}=\sum_{i=1}^sv_i^{-1}e_i$. By
\cite[Theorem 3.2.1]{webster}, ${\rm IP}(\mathfrak{q})\textstyle\bigcap
H_i$ is a facet of ${\rm IP}(\mathfrak{q})$ for all $1\leq i\leq
\ell$. Note that $H_{e_i}\bigcap{\rm IP}(I)\neq
H_{(v^{-1},1)}\bigcap{\rm IP}(I)$ for all $1\leq i\leq s$ because
$v_ie_i\in{\rm NP}(I)$, $\langle v_ie_i,v^{-1}\rangle=1$ and $\langle
v_ie_i,e_i\rangle\neq 0$. 
Hence, using the equality ${\rm NP}(\mathfrak{q})={\rm
IP}(\mathfrak{q})$ and Eq.~\eqref{jun4-21-1}, we obtain that
$$
{\rm NP}(\mathfrak{q})\textstyle\bigcap H_i={\rm NP}(\mathfrak{q})\bigcap\{x\mid
\langle x,v^{-1}\rangle=1\}
$$
for some $1\leq i\leq\ell$.  Since $v_1e_1,\ldots,v_se_s$ are in the right hand side of this
equality it follows that $\alpha_i=(v_1,\ldots,v_s)$. Hence
$t^{v_q}$ is in
$\mathfrak{q}_{\alpha_i}=(t_1^{v_1},\ldots,t_s^{v_s})$, a
contradiction because $t^{v_q}\in G(\mathfrak{q})$.

Case (II): The vertices of ${\rm NP}(\mathfrak{q})$ are not contained
in $\{v_1e_1,\ldots,v_se_s\}$. Then, $v_k$ is a vertex of ${\rm
NP}(\mathfrak{q})$ for some $k>s$. Hence, by
\cite[Corollary~1.1.49]{monalg-rev}, $v_k$ is a basic feasible
solution in the sense of \cite[Definition~1.1.48]{monalg-rev} for the
system ``$x\geq 0$, $\langle x,\alpha_i^{-1}\rangle\geq 1$,
$i=1,\ldots,\ell$'' of linear
constraints that represent ${IP}(I)$. Therefore there is $1\leq i\leq \ell$ such that
$v_k\in H_i\bigcap{\rm IP}(\mathfrak{q})$. Thus, writing
$v_k=(v_{k,1},\ldots,v_{k,s})$ and
$\alpha_i=(\alpha_{i,1},\ldots,\alpha_{i,s})$, one has
\begin{equation}\label{jun5-21}
({v_{k,1}}/{\alpha_{i,1}})+\cdots+({v_{k,s}}/{\alpha_{i,s}})=1.
\end{equation} 
\quad As $t^{v_k}$ is in
$\mathfrak{q}_{\alpha_i}=(t_1^{\alpha_{i,1}},\ldots,t_s^{\alpha_{i,s}})$,
we get $v_{k,j}\geq\alpha_{i,j}$ for some $1\leq j\leq s$. Hence, by
Eq.~\eqref{jun5-21}, we get $v_{k,p}=0$ for $p\neq j$ 
and $v_{k,j}=\alpha_{i,j}$. Thus $v_k=\alpha_{i,j}e_j$. Since
$t^{v_k}\in G(\mathfrak{q})$, one has
$\alpha_{i,j}=v_{k,j}<v_j$, and consequently
$t_j^{v_j}\in(t_j^{\alpha_{i,j}})=(t^{v_k})$, a contradiction because
$t_j^{v_j}\in G(\mathfrak{q})$.

$\Leftarrow$): This follows from Lemma~\ref{irreducible-ideals}. 
\end{proof}

\begin{proposition}\label{normal-irreducible-ideals}
Let $\mathfrak{q}=(t_1^{b_1},\ldots,t_r^{b_r})$ be an irreducible
monomial ideal of $S$, $b_i\geq 1$ for all $i$. The following
conditions are equivalent.
\begin{enumerate}
\item[(a)] $\mathfrak{q}$ is normal.\quad {\rm (b)} $\mathfrak{q}$ is complete. 
\item[(c)]
$\mathfrak{q}=(t_1,\ldots,t_{j-1},t_j^{b_j},t_{j+1},\ldots,t_r)$ for
some $j$, that is, $b_i=1$ for $i\in\{1,\ldots,r\}\setminus\{j\}$.
\end{enumerate} 
\end{proposition}

\begin{proof} (a) $\Rightarrow$ (b): This implication is clear because
all powers of $\mathfrak{q}$ are complete.

(b) $\Rightarrow$ (c): We proceed by contradiction assuming that there
are $b_i$ and $b_j$, $1\leq i<j\leq r$, such that $b_i\geq 2$ and $b_j\geq 2$.
We may assume $b_i\geq b_j$. Then the vector $(b_i-1)e_i+e_j$
satisfies the linear inequality
$$
b_1^{-1}x_1+\cdots+b_r^{-1}x_s\geq 1
$$
because $((b_i-1)/b_i)+(1/b_j)\geq 1$. Thus $(b_i-1)e_i+e_j$ is in
${\rm IP}(\mathfrak{q})$. Hence, using
Lemma~\ref{irreducible-ideals}, one has ${\rm
NP}(\mathfrak{q})={\rm IP}(\mathfrak{q})$. Therefore, by
Proposition~\ref{np-qa}, $t_i^{b_i-1}t_j$ is in
$\overline{\mathfrak{q}}=\mathfrak{q}$, a contradiction because 
$t_i^{b_i-1}t_j$ is not a multiple of $t_i^{b_i}$ or $t_j^{b_j}$.

(c) $\Rightarrow$ (a):  For simplicity of notation we may assume $j=1$,
that is, $b_1\geq 1$ and $b_i=1$ for $i=2,\ldots,r$. 
To show that
$\mathfrak{q}$ is normal it suffices to show the inclusion
$\overline{\mathfrak{q}^n}\subset\mathfrak{q}^n$ for $n\geq 1$. Take
$t^a$ a minimal generator of $\overline{\mathfrak{q}^n}$ with
$a=(a_1,\ldots,a_s)$ and $a_i=0$ for $i>r$. Then, by Proposition~\ref{np-qa},
$a/n\in{\rm NP}(\mathfrak{q})$. Using Lemma~\ref{irreducible-ideals}, one has ${\rm
NP}(\mathfrak{q})={\rm IP}(\mathfrak{q})$. Hence $a/n$ satisfies the
inequality 
$$
b_1^{-1}x_1+x_2+\cdots+x_r\geq 1,
$$
that is, $b_1^{-1}a_1+a_2+\cdots+a_r\geq n$. By the division
algorithm one can write $a_1=k_1b_1+r_1$, where
$k_1,r_1\in\mathbb{N}$ and $0\leq r_1<b_1$. Then
$$
b_1^{-1}(k_1b_1+r_1)+a_2+\cdots+a_r=k_1+{b_1^{-1}}{r_1}+a_2+\cdots+a_r\geq
n.
$$ 
\quad If $n-1\geq k_1+a_2+\cdots+a_r$, using the inequality above, 
we obtain $b_1^{-1}r_1\geq 1$, a contradiction. Thus
$k_1+a_2+\cdots+a_r\geq n$. Writing $t^a$ as
$$
t^a=t_1^{a_1}\cdots t_r^{a_r}=((t_1^{b_1})^{k_1}t_2^{a_2}\cdots
t_r^{a_r})t_1^{r_1},
$$
we obtain that $t^a$ is in $\mathfrak{q}^n$.
\end{proof}

\begin{theorem}\label{NP-IP-char}
Let $I$ be a monomial ideal of $S$ and let 
$I=\mathfrak{q}_1\bigcap\cdots\bigcap\mathfrak{q}_m$ be the
irreducible decomposition of $I$. Then, ${\rm NP}(I)={\rm IP}(I)$ if
and only if
$\overline{I^n}=\overline{\mathfrak{q}_1^n}\bigcap\cdots\bigcap\overline{\mathfrak{q}_m^n}$
for all $n\geq 1$.
\end{theorem}

\begin{proof} $\Rightarrow$): The inclusion ``$\subset$'' is clear
because $\overline{I^n}\subset\overline{\mathfrak{q}_i^n}$ for all
$i$. To show the inclusion ``$\supset$'' take
$t^a\in\overline{\mathfrak{q}_i^n}$ 
for all $i$. Then, by Proposition~\ref{np-qa}, $a/n\in {\rm
NP}(\mathfrak{q}_i)$ for all $i$. Hence, by
Lemma~\ref{irreducible-ideals}, $a/n\in
{\rm IP}(\mathfrak{q}_i)$ for all $i$. Thus, by construction of ${\rm
IP}(I)$, we get $a/n\in {\rm IP}(I)=
{\rm NP}(I)$. Then, by Proposition~\ref{np-qa},
$t^a\in\overline{I^n}$.

$\Leftarrow$): The inclusion ${\rm
NP}(I)\subset{\rm IP}(I)$ holds in general
\cite[Theorem~3.7]{Seceleanu-convex-bodies}. Note that this
inclusion follows from Lemma~\ref{irreducible-ideals} and Eq.~\eqref{NP-def} by
using a generating set for $I$ and the
equality $I=\textstyle\bigcap_{i=1}^m\mathfrak{q}_i$. To show the
inclusion ${\rm
NP}(I)\supset{\rm IP}(I)$ take $a\in{\rm IP}(I)$. As ${\rm IP}(I)$ is
a rational polyhedron of blocking type (Lemma~\ref{blocking-type}), we may assume that 
$0\neq a\in\mathbb{Q}_+^s$. Then there is $0\neq n\in\mathbb{N}$ such that
$na\in\mathbb{N}^s$. Using Lemma~\ref{irreducible-ideals}, we get
$$
n{\rm IP}(I)\subset n{\rm IP}(\mathfrak{q}_i)=n{\rm NP}(\mathfrak{q}_i)
$$  
for all $i$. Setting $b=na$, we obtain $b\in n{\rm NP}(\mathfrak{q}_i)$ for all
$i$. Hence, by Proposition~\ref{np-qa},
$t^b\in\overline{\mathfrak{q}_i^n}$ for all $i$. Thus, by hypothesis
and Proposition~\ref{np-qa}, one has $t^b\in\overline{I^n}$ and
$b/n=a$ is in ${\rm NP}(I)$. 
\end{proof}

\begin{proposition}\label{jan19-21} Let $I$ be a monomial ideal of $S$, 
let $I=\mathfrak{q}_1\bigcap\cdots\bigcap\mathfrak{q}_m$ be the
irreducible decomposition of $I$, let $\alpha_i$ 
be the vector in $\mathbb{N}^s\setminus\{0\}$ such that 
$\mathfrak{q}_i=\mathfrak{q}_{\alpha_i}$ for $i=1,\ldots,m$, and let 
$B$ be the $s\times m$ matrix with column vectors
$\alpha_1^{-1},\ldots,\alpha_m^{-1}$. The
following hold.
\begin{enumerate}
\item[(a)] A monomial $t^a$ is in
$\overline{\mathfrak{q}_1^n}\bigcap\cdots\bigcap\overline{\mathfrak{q}_m^n}$
if and only if $a/n$ is in $\mathcal{Q}(B)$.
\item[(b)] If ${\rm rad}(\mathfrak{q}_j)\neq{\rm
rad}(\mathfrak{q}_i)$ for $i\neq j$ 
and the isolated components
$\mathfrak{q}_1,\ldots,\mathfrak{q}_r$ of $I$ are normal, then a
monomial $t^a$ is in $I^{(n)}$ if and only if $\langle
a/n,\alpha_i^{-1}\rangle\geq 1$ for $i=1,\ldots,r$. If in addition we
assume that $I$ has no embedded primes, then $t^a\in I^{(n)}$ if and
only if $a/n\in\mathcal{Q}(B)$.
\item[(c)] If ${\rm NP}(I)={\rm IP}(I)$, then $t^a\in \overline{I^n}$
if and only if $a/n\in\mathcal{Q}(B)$. If in addition we assume that 
$I$ has no embedded
primes, ${\rm rad}(\mathfrak{q}_j)\neq{\rm
rad}(\mathfrak{q}_i)$ for $i\neq j$, and $\mathfrak{q}_i$ is normal for all $i$, 
then $\overline{I^n}=I^{(n)}$ for all $n\geq 1$.
\end{enumerate}
\end{proposition}

\begin{proof} (a): By Lemma~\ref{irreducible-ideals}, $t^a$ is in 
$\overline{\mathfrak{q}_1^n}\bigcap\cdots\bigcap\overline{\mathfrak{q}_m^n}$
if and only if $\langle a/n,\alpha_i^{-1}\rangle\geq 1$ for
$i=1,\ldots,m$, that is, if and only if $a/n\in\mathcal{Q}(B)$.

(b): The $n$-th symbolic power of $I$ is given by 
$I^{(n)}=\mathfrak{q}_1^n\bigcap\cdots\bigcap\mathfrak{q}_r^n$
(Lemma~\ref{anoth-one-char-spow-general}). Since
$\mathfrak{q}_i$ is normal for $i=1,\ldots,r$, we obtain that $t^a$ is
in $I^{(n)}$ if and only if $t^a$ is in
$\overline{\mathfrak{q}_1^n}\bigcap\cdots\bigcap\overline{\mathfrak{q}_r^n}$.
Hence, By Lemma~\ref{irreducible-ideals}, $t^a$ is
in $I^{(n)}$ if and only if $\langle a/n,\alpha_i^{-1}\rangle\geq 1$
for $i=1,\ldots,r$. In particular, if $I$ has no embedded primes, that
is, $r=m$, one has that $t^a\in I^{(n)}$ if and
only if $a/n\in\mathcal{Q}(B)$.

(c): By Theorem~\ref{NP-IP-char}, $t^{a}$ is in $\overline{I^n}$ if
and only if $t^a$ is in
$\overline{\mathfrak{q}_1^n}
\bigcap\cdots\bigcap\overline{\mathfrak{q}_m^n}$. Thus, by part (a), 
$t^a$ is in $\overline{I^n}$ if and only if $a/n\in\mathcal{Q}(B)$.
Therefore, under the additional assumptions, using part (b), we obtain 
$\overline{I^n}=I^{(n)}$ for all $n\geq 1$.
\end{proof}

\begin{theorem}\label{np=ip-char} Let $I$ be a monomial ideal of $S$
and let $I=\mathfrak{q}_1\bigcap\cdots\bigcap\mathfrak{q}_m$ be its 
irreducible decomposition. Suppose that $I$ has no embedded
associated primes, ${\rm rad}(\mathfrak{q}_j)\neq{\rm rad}(\mathfrak{q}_i)$ for $j\neq
i$ and $\mathfrak{q}_i$ is normal for all $i$. The following conditions
are equivalent.
\begin{enumerate}
\item[(a)] ${\rm IP}(I)$ is integral.\quad {\rm(b)} ${\rm NP}(I)={\rm
IP}(I)$.\quad {\rm (c)} $\overline{I^n}=I^{(n)}$ for all $n\geq 1$.
\end{enumerate}
\end{theorem}

\begin{proof} (a) $\Rightarrow$ (b): The inclusion ${\rm
NP}(I)\subset{\rm IP}(I)$ holds in general 
\cite[Theorem~3.7]{Seceleanu-convex-bodies}. Let $V$ be the vertex set
of ${\rm IP}(I)$. We claim that $V\subset{\rm NP}(I)$. Take $a\in V$.
As $a$ is integral, by Lemma~\ref{irreducible-ideals}, $t^a\in{\rm
NP}(\mathfrak{q}_i)$ for all $i$. Thus, 
$t^a\in\overline{\mathfrak{q}_i}=\mathfrak{q}_i$ for all $i$, 
$t^a\in I$, and $a\in {\rm NP}(I)$. This proves the claim. Therefore,
by Proposition~\ref{cpr}, we obtain
$$
{\rm IP}(I)=\mathbb{R}_+^s+{\rm conv}(V)\subset
\mathbb{R}_+^s+{\rm NP}(I)={\rm NP}(I).
$$
\quad (b) $\Rightarrow$ (c): This implication follows at once from
Proposition~\ref{jan19-21}(c).

(c) $\Rightarrow$ (a):  Since ${\rm NP}(I)$ is integral, we need only
show ${\rm NP}(I)={\rm IP}(I)$. Note that
$$
\overline{I^n}=I^{(n)}={\mathfrak{q}_1^n}\textstyle\bigcap\cdots\bigcap{\mathfrak{q}_m^n}
=\overline{\mathfrak{q}_1^n}\bigcap\cdots\bigcap\overline{\mathfrak{q}_m^n}
$$
for all $n\geq 1$ (Lemma~\ref{anoth-one-char-spow-general}). Hence,
by Theorem~\ref{NP-IP-char}, we obtain ${\rm 
NP}(I)={\rm IP}(I)$.
\end{proof}

\begin{remark}\label{aug7-21} Let $I$ be a squarefree monomial ideal
of $S$. Then by \cite[Corollary~4.16]{Francisco-TAMS},
$\widehat{\rho}(I)\geq 1$, with equality if and only if
$I^{(n)}\subset\overline{I^n}$ for every $n\geq 1$. Therefore, by
Corollary~\ref{ntf-char}, $\widehat{\rho}(I)=1$ if and only if
$\mathcal{Q}(I)$ is integral. 
\end{remark}

\section{Examples}\label{examples-section}

\begin{example}\label{non-normal} Let $S=\mathbb{Q}[t_1,t_2,t_3]$ be a
polynomial ring and let $I=(t_1t_2^2,\,t_2t_3^2,\,t_1t_3^2)$ be the monomial
ideal whose incidence matrix is
$$
A=\left[\begin{matrix}
1&0&1\cr
2&1&0\cr
0&2&2
\end{matrix}
\right].
$$
\quad Using Procedure~\ref{non-normal-procedure} we obtain that the vertices of
$\mathcal{Q}(I)$ are 
$$
(0, 1/2, 1/2),\, (1,0, 1/2),\, (1,1,0),\, (1/3, 1/3, 1/3). 
$$
\quad The irreducible decomposition of $I$
is minimal because $I$ is the edge ideal of a weighted oriented graph
\cite[Theorem~25]{WOG}. The minimal primes of $I$ are
$\mathfrak{p}_1=(t_2,t_3)$, $\mathfrak{p}_2=(t_1,t_3)$ and
$\mathfrak{p}_3=(t_1,t_2)$. Let $\mathfrak{q}_i$ be the irreducible
component of $I$ corresponding to $\mathfrak{p}_i$. Then, by
Theorem~\ref{irreducible-deco-qA}(a), the 
first three vertices of $\mathcal{Q}(A)$ correspond to $\mathfrak{p}_1$, $\mathfrak{p}_2$,
$\mathfrak{p}_3$, respectively, and
we have the equality $I=(t_2^2,t_3^2)\bigcap(t_1,t_3^2)\bigcap(t_1,t_2)$.
\end{example}

\begin{example}\label{embedded-primes} 
Let $S=\mathbb{Q}[t_1,t_2,t_3,t_4]$ be a
polynomial ring and let $A$ be the incidence matrix of the monomial
ideal $I=(t_1t_2,\,t_3t_4^3,\,t_1t_3t_4^2,\,
t_2t_3^3,t_3^3t_4^2)$.
Adapting Procedure~\ref{non-normal-procedure} we obtain that the
vertices of $\mathcal{Q}(A)$ are 
\begin{align*}
&(0,   1,   1,   0),\,(0,\,   1,\,
0,\, 1/2)\, ,\, (1,   0,   1,   0),\,(1,0, 1/3, 2/9),\\
&(2/7,\,5/7,\, 1/7,\, 2/7),\, (3/7,\, 4/7,\, 1/7,\, 2/7).
\end{align*}
\quad The irreducible decomposition of the ideal is
$I=(t_2,t_3)\bigcap(t_2,t_4^2)\bigcap(t_1,t_3)\bigcap(t_1,t_3^3,t_4^3)$
and $(t_1,t_3,t_4)$ is an embedded associated prime of $I$.  
Then, by Theorem~\ref{irreducible-deco-qA}(a), the first three vertices 
of $\mathcal{Q}(A)$ listed above determine the irreducible components
of $I$ corresponding to the minimal primes
$\mathfrak{p}_1=(t_2,t_3)$, $\mathfrak{p}_2=(t_2,t_4)$ and
$\mathfrak{p}_3=(t_1,t_3)$.
\end{example}

\begin{example}\label{covering-algebras} 
Let $S=\mathbb{Q}[t_1,t_2,t_3]$ be a
polynomial ring, let $\mathcal{Q}(C)$ be the covering polyhedron of the matrix $C=(1/2,\, 1/5,\,
1/11)^\top$, and let $\mathcal{F}=\{I_n\}_{n=0}^\infty$ be the
filtration associated to $\mathcal{Q}=\mathcal{Q}(C)$. 
Using Theorem~\ref{hilbert-basis-filtration} and
Procedure~\ref{covering-algebras-procedure}, we obtain
$$  
I_1=(t_1^2,\,t_2^5,\,t_3^{11},\,
t_2t_3^9,\,
t_2^2t_3^7,\,
t_2^3t_3^5,\, 
t_2^4t_3^3,\,
t_1t_3^6,\, 
t_1t_2t_3^4,\, 
t_1t_2^2t_3^2,\, 
t_1t_2^3),
$$
and the Rees algebra $\mathcal{R}(\mathcal{F})$ of $\mathcal{F}$ is
$S[I_1z,\, 
t_1t_2^3t_3^{10}z^2,\, 
t_1t_2^4t_3^8z^2]$. As $I_1$ is a primary ideal, we get $\mathcal{R}(I_1)=\mathcal{R}_s(I_1)
\subsetneq\mathcal{R}(\mathcal{F})$. In this example $\alpha_{\mathcal{F}}(1)=2$,
the bigheight $e$ of $I_1$ (i.e., $e$ is the largest height of an
associated prime of $I_1$)  is $3$, and the vertex set of $\mathcal{Q}$
is equal to $V(\mathcal{Q})=\{2e_1,\, 5e_2,\, 11e_3\}$. Thus, $\alpha(\mathcal{Q})=
\min\{|v|\colon\,v\in V(\mathcal{Q})\}=2$. By
Corollary~\ref{schrijver-constant}, $\alpha(\mathcal{Q})$ is the
Waldschmidt constant $\widehat{\alpha}(\mathcal{F})$ of the filtration
$\mathcal{F}$, that is
$$
\alpha(\mathcal{Q})=\widehat{\alpha}(\mathcal{F})=
\lim_{n\rightarrow\infty}{\alpha_{\mathcal{F}}(n)}/{n}=2.
$$
\quad Thus in this case $\widehat{\alpha}(\mathcal{F})=
(\alpha_{\mathcal{F}}(1)+e-1)/e$ (cf. \cite[Theorem~5.3]{Waldschmidt-Bocci-etal}). 
\end{example}

\begin{example}\label{non-strict-filtration}
Let $S=\mathbb{Q}[t_1,t_2]$ be a
polynomial ring, let $\mathcal{Q}(C)$ be the covering polyhedron of
the matrix $C=(3/2,\, 3/2)^\top$, and let $\mathcal{F}=\{I_n\}_{n=0}^\infty$ be the
filtration associated to $\mathcal{Q}(C)$. By 
Theorem~\ref{hilbert-basis-filtration} and adapting
Procedure~\ref{covering-algebras-procedure}, we get that $\mathcal{R}(\mathcal{F})$ is 
$\mathbb{Q}[t_1,t_2,t_1t_2z,t_1^2t_2^2z^3]$, $I_i$, $i=1,2,3$, are given by 
$$  
I_1=(t_1t_2),\, I_2=I_3=(t_1^2t_2^2),
$$
and $\mathcal{F}=\{I_n\}_{n=0}^\infty$ is not strict. The only vertex of $\mathcal{Q}(C)$ is $(2/3,2/3)$
(cf. Lemma~\ref{filtration}(b)).
\end{example}

\begin{example}[Irreducible filtration]\label{filtration3} 
Let $I$ be a monomial ideal of $S$. 
The filtration of the irreducible polyhedron
of $I$ is constructed as follows. Let $I=\mathfrak{q}_1\bigcap\cdots\bigcap\mathfrak{q}_m$ be the
irreducible decomposition of $I$, let $\alpha_i$ 
be the vector in $\mathbb{N}^s\setminus\{0\}$ such that 
$\mathfrak{q}_i=\mathfrak{q}_{\alpha_i}$ for $i=1,\ldots,m$, and let 
$B$ be the $s\times m$ matrix with column vectors
$\alpha_1^{-1},\ldots,\alpha_m^{-1}$ (Section~\ref{section-intro}).
The covering polyhedron $\mathcal{Q}(B)$ is 
${\rm IP}(I)$, the {irreducible polyhedron} of $I$. Then, by
Proposition~\ref{jan19-21}(a), one has
$$
\overline{\mathfrak{q}_1^n}\textstyle\bigcap\cdots\bigcap\overline{\mathfrak{q}_m^n}=
(\{t^a\vert\, a/n\in\mathcal{Q}(B)\}),\ n\geq 1,
$$
that is, the filtration
$\mathcal{F}=\{I_n\}_{n=0}^\infty$ associated to ${\rm IP}(I)$ is
given by $I_n=\bigcap_{i=1}^m\overline{\mathfrak{q}_i^n}
$ for $n\geq 1$ and $I_0=S$. Thus,
$I_1=\bigcap_{i=1}^m\overline{\mathfrak{q}_i}$ and 
$\mathcal{R}(\mathcal{F})=\bigoplus_{n=0}^\infty I_nz^n$ is the Rees
algebra of $\mathcal{F}$. 
\end{example}

\begin{example}\label{bowtie-example}
Let $S=\mathbb{Q}[t_1,\ldots,t_7]$ be a polynomial ring and let
$I=I(G)$ be the edge ideal of the graph $G$ of Figure.~\ref{figure1}.
\begin{figure}[ht]
\vspace{1cm}
\setlength{\unitlength}{.035cm}
\thicklines
\begin{picture}(10,5)
\put(-40,0){\circle*{3.6}}
\put(-40,6){$t_1$}
\put(0,0){\circle*{3.6}}
\put(0,6){$t_4$}
\put(40,0){\circle*{3.6}}
\put(36,6){$t_5$}
\put(80,20){\circle*{3.6}}
\put(85,20){$t_6$}
\put(80,-20){\circle*{3.6}}
\put(85,-20){$t_7$}
\put(-80,20){\circle*{3.6}}
\put(-80,-20){\circle*{3.6}}
\put(40,0){\line(2,-1){40}}
\put(40,0){\line(2,1){40}}
\put(80,-20){\line(0,1){40}}
\put(-95,20){$t_2$}
\put(-80,-20){\line(0,1){40}}
\put(-95,-20){$t_3$}
\put(-40,0){\line(-2,-1){40}}
\put(-40,0){\line(-2,1){40}}
\put(0,0){\line(1,0){40}}
\put(0,0){\line(-1,0){40}}
\end{picture}
\vspace{1cm}
\caption{Graph $G$ with non-normal edge ideal.}\label{figure1}
\end{figure}
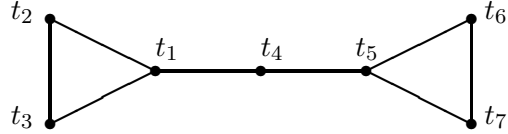
This ideal is not normal because $t_1t_2t_3t_5t_6t_7$ is in
$\overline{I\sp 3}\setminus I\sp 3$. Let $C$ be the incidence matrix
of $I^\vee=I_c(G)$ whose transpose is the following matrix 
$$
C^\top=\left[\begin{matrix}
0& 1& 1& 1& 0& 1& 1\cr 
0&1&1&1&1&0&1\cr
0&1&1&1&1&1&0\cr
1&0&1&0&1&0&1\cr
1&0&1&0&1&1&0\cr
1&0&1&1&0&1&1\cr
1&1&0&0&1&0&1\cr 
1&1&0&0&1&1&0\cr 
1&1&0&1&0&1&1 
\end{matrix}
\right].
$$
\quad The rows of $C^\top$ correspond to the minimal vertex covers 
of $G$ and also correspond to the associated primes of $I$. The
covering polyhedron of $I_c(G)$ is $\mathcal{Q}(C)$ and this is the symbolic polyhedron of 
$I$. The filtration
$\mathcal{F}=\{I_n\}_{n=0}^\infty$ of $\mathcal{Q}(C)$ is the filtration of symbolic powers
of $I$. In this case $\rho_{ic}(\mathcal{F})$ is $\rho_{ic}(I)$, the ic-resurgence
of $I$. Recall that $\rho_{ic}(I)$ is the
asymptotic resurgence $\widehat{\rho}(I)$ of $I$
\cite[Corollary~4.14]{Francisco-TAMS}. 
From \cite[Theorem~3.12, Corollary~4.14]{Francisco-TAMS}, 
\cite[Theorem~6.7(i)]{Waldschmidt-Bocci-etal}, and
Theorem~\ref{jun14-21}, one has
$$
\rho_{ic}(I)={2(\omega(G)-1)}/{\omega(G)}=\rho_{ic}(I_c(G))=4/3
$$ 
because $G$ is chordal, hence perfect
\cite[Proposition~5.5.2]{diestel}. Using \textit{Macaulay}$2$
\cite{mac2}, it is seen that $I^{(4)}\subset
I^3+(t_1t_2t_3t_5t_6t_7)\subset\overline{I^3}$. We will use this example to illustrate
our algorithm  to compute $\rho_{ic}(I)$ 
(Procedure~\ref{bowtie-procedure}). The input for this algorithm is
the incidence
matrix of $I$. For squarefree monomial ideals, an easy way to compute
$\rho_{ic}(I)$ is to use Algorithm~\ref{AS-code}. 
\end{example}

\begin{appendix}

\section{Procedures}\label{Appendix}

In this appendix we give procedures for \textit{Normaliz} \cite{normaliz2}, \textit{PORTA} \cite{porta}, 
and \textit{Macaulay}$2$ \cite{mac2} that are used in the
examples of Section~\ref{examples-section} and present an algorithm to
compute the asymptotic resurgence of any squarefree monomial ideal. 

\begin{procedure}\label{non-normal-procedure}
Let $I$ be a monomial ideal and let
$\mathcal{Q}(I)$ be its covering
polyhedron. This procedure for
\textit{PORTA} \cite{porta} computes the vertices of $\mathcal{Q}(I)$ of
Example~\ref{non-normal}.

\begin{verbatim}
DIM = 3
VALID
7 7 7 
INEQUALITIES_SECTION
x1+2x2>=1
x2+2x3>=1
x1+2x3>=1
x1>=0
x2>=0
x3>=0
END
\end{verbatim}
\end{procedure}

\begin{procedure}\label{covering-algebras-procedure}
Let $\mathcal{Q}(C)$ be the covering polyhedron of a matrix
$C$ and let $\mathcal{F}=\{I_n\}_{n=0}^\infty$
be its  associated filtration. This procedure for
\textit{Normaliz} \cite{normaliz2} computes the Hilbert basis 
of the Simis cone ${\rm SC}(\mathcal{Q}(C))$ of $\mathcal{Q}(C)$ of Example~\ref{covering-algebras}.
Using Theorem~\ref{hilbert-basis-filtration}, we
obtain a finite generating set for
the Rees algebra $\mathcal{R}(\mathcal{F})$ of the filtration
$\mathcal{F}$. 

\begin{verbatim}
/* This computes the Hilbert basis 
of the Simis cone of a covering polyhedron */
amb_space auto
inequalities
[[1/2 1/5 1/11 -1] 
[1 0 0 0]
[0 1 0 0]
[0 0 1 0]
[0 0 0 1]]
\end{verbatim}
\end{procedure}


\begin{procedure}[Algorithm for the ic-resurgence]\label{bowtie-procedure}
Let $S=K[t_1,\ldots,t_s]$ be a polynomial ring over the field
$K=\mathbb{Q}$ and let $I\subset S$ be a squarefree monomial ideal of
height at least $2$, let
$G(I)=\{t^{v_1},\ldots,t^{v_q}\}$ be the minimal set of generators of
$I$ and let $\mathcal{F}=\{I^{(n)}\}_{n=0}^\infty$ be the filtration of
symbolic powers. Recall that $\mathcal{F}$ is the filtration associated
to the covering polyhedron of $I^\vee$. 
In this case the ic-resurgence of $\mathcal{F}$ is 
$\rho_{ic}(I)$, the ic-resurgence of $I$. This procedure gives an
 algorithm, using
\textit{Normaliz} \cite{normaliz2}, to compute
this number.  To compute
$\rho_{ic}(I)$, we need an efficient way to determine the polyhedron
$\mathcal{P}_j$ of feasible 
points of Theorem~\ref{lp-resurgence-formula} for each $1\leq j\leq
k$. 

 The incidence matrix of $I$, denoted by $A$,
is the matrix with column vectors $v_1,\ldots,v_q$. 
Our algorithm is based on the computation of the support hyperplanes
of the \textit{Rees cone} ${\rm RC}(I)$ of $I$ which is the
finitely generated rational cone defined as \cite{normali}: 
\begin{equation}\label{rees-cone-eq}
{\rm RC}(I):=\mathbb{R}_+\{e_1,\ldots,e_s,(v_1,1),\ldots,(v_q,1)\}.
\end{equation}
\quad This is a rational polyhedral cone that has a unique irreducible representation 
\begin{equation}\label{supp-hyp}
{\rm RC}(I)=\left(\bigcap_{i=1}^{s+1}H^+_{e_i}\right)\bigcap
\left(\bigcap_{i=1}^mH^+_{(u_i,-1)}\right)
\bigcap\left(\bigcap_{i=1}^\ell H^+_{(\gamma_i,-d_i)}\right),
\end{equation}
where none of the closed halfspaces can be omitted from the
intersection, $u_1,\ldots,u_m$ are the exponent vectors of the 
minimal generators of $I^\vee$, $\gamma_i\in\mathbb{N}^s\setminus\{0\}$, 
$d_i\in\mathbb{N}\setminus\{0,1\}$, and the non-zero entries of
$(\gamma_i,-d_i)$ are relatively prime
\cite[Proposition~1.1.51, Theorem~14.1.1]{monalg-rev}. The hyperplanes 
defining the closed halfspaces of Eq.~\eqref{supp-hyp} are the
\textit{support hyperplanes} of the Rees cone of $I$. 
Setting $\beta_i=u_i$ for $i=1,\ldots,m$, $\beta_{m+i}=\gamma_i/d_i$
for $i=1,\ldots,\ell$, and $k=m+\ell$, by \cite[Theorem~3.1]{reesclu}
and Proposition~\ref{np-qa}, one has that the Newton polyhedron 
${\rm NP}(I)$ of $I$ is equal to the covering polyhedron
$\mathcal{Q}(B)$, where $B$ is the matrix with column vectors
$\beta_1,\ldots,\beta_k$. Thus, for each $1\leq j\leq k$, we can determine the
polyhedron $\mathcal{P}_j$ of feasible points of
Theorem~\ref{lp-resurgence-formula} 
by setting $c_i=u_i$ and $n_i=1$ for $i=1,\ldots,m$, and $n_{m+i}=d_i$ for
$i=1,\ldots,\ell$. The vertices
of $\mathcal{Q}(A)$ are precisely $\beta_1,\ldots,\beta_k$
\cite[Theorem~3.1]{reesclu}, that is, finding the support hyperplanes
of ${\rm RC}(I)$ is equivalent to finding the vertices of
$\mathcal{Q}(A)$. 

We illustrate our algorithm with the edge ideal $I$ of 
Example~\ref{bowtie-example}. This ideal is given by   
$$ 
I=(t_1t_2,\, t_2t_3,\, t_1t_3,\, t_1t_4,\, t_4t_5,\, t_5t_6,\,
t_6t_7,\, t_5t_7).
$$
\quad The first step is to put the transpose of $A$ 
in the following input file for \textit{Normaliz}. The rows of $A^\top$
correspond to the exponent vectors $v_1,\ldots,v_q$ of the monomials in $G(I)$.
\begin{verbatim}
/* This computes the integral closure 
of the Rees algebra and the support hyperplanes 
of the Rees cone */
amb_space 8
rees_algebra 8
1 1 0 0 0 0 0 
0 1 1 0 0 0 0 
1 0 1 0 0 0 0
1 0 0 1 0 0 0
0 0 0 1 1 0 0
0 0 0 0 1 1 0
0 0 0 0 0 1 1
0 0 0 0 1 0 1
\end{verbatim}
\quad Using \textit{Normaliz} we obtain that the monomial $t_1t_2t_3t_5t_6t_7$ is in
$\overline{I\sp 3}\setminus I\sp 3$, that is, $I$ is not normal, and
we also obtain the following full list of support hyperplanes of
${\rm RC}(I)$: 

\begin{verbatim}
24 support hyperplanes:
 0 0 0 0 0 0 0  1
 0 0 0 0 0 0 1  0
 0 0 0 0 0 1 0  0
 0 0 0 0 1 0 0  0
 0 0 0 1 0 0 0  0
 0 0 1 0 0 0 0  0
 0 1 0 0 0 0 0  0
 0 1 1 1 0 1 1 -1
 0 1 1 1 1 0 1 -1
 0 1 1 1 1 1 0 -1
 0 2 2 2 1 1 1 -2
 1 0 0 0 0 0 0  0
 1 0 1 0 1 0 1 -1
 1 0 1 0 1 1 0 -1
 1 0 1 1 0 1 1 -1
 1 1 0 0 1 0 1 -1
 1 1 0 0 1 1 0 -1
 1 1 0 1 0 1 1 -1
 1 1 1 1 1 1 1 -2
 1 1 1 1 2 0 2 -2
 1 1 1 1 2 2 0 -2
 1 1 1 2 0 2 2 -2
 2 0 2 1 1 1 1 -2
 2 2 0 1 1 1 1 -2
\end{verbatim}

For each row of this matrix that corresponds to a vertex of
$\mathcal{Q}(A)$, this
matrix will be 
transformed---by successively applying the
following rules---into a set of linear constraints that define 
a polyhedron $\mathcal{P}_j$ of feasible points of Theorem~\ref{lp-resurgence-formula}. The constraints will
be in a format that can be read by \textit{Normaliz} \cite{normaliz2}.
The rules to obtain the linear constraints are:

R1) Fix any row vector $a_1\cdots a_{s+1}$, $s=7$, of this matrix
with $a_{s+1}\notin\{0,1\}$ and add the linear constraint 
$a_1\cdots a_s\ 0 \ a_{s+1}\ 1\ <=\ 0$ at the end of this matrix. We fix the row $1\ 1\ 1\ 1\
1\ 1\ 1\,-2$ that
corresponds to the vertex  $1/2(1, 1, 1, 1, 1, 1, 1)$ of the
covering polyhedron $\mathcal{Q}(A)$. 

R2) Remove any row $b_1\cdots b_{s+1}$ whose last entry $b_{s+1}$ is not in $\{0,-1\}$.

R3) Replace any row $b_1\cdots b_{s+1}$ with $b_{s+1}=-1$ or
$b_{s+1}=0$ by
the linear constraint $b_1\cdots b_{s+1}\ 0\ 0\ >=\ 0$.

R4) Add the constraints $0\cdots 0\ 1\ 0\ =\ 1$, 
$0\cdots 0\ 1\ 0 \ -1\ >=\ 0$ and $0\cdots 0\ 1\ >=\ 0$, where the left hand
side of each of these constraints has $s+3$ digits.

Applying these rules successively to the matrix of support hyperplanes
we obtain the following input file for \textit{Normaliz} \cite{normaliz2}: 
 
\begin{verbatim}
/* This is one of the polyhedrons of feasible 
points of the linear programs that are used to 
compute the ic-resurgence */
amb_space 10
constraints 20
 0 0 0 0 0 0 0 0 1 0 = 1
 0 1 1 1 0 1 1 -1 0 0 >= 0
 0 1 1 1 1 0 1 -1 0 0 >= 0
 0 1 1 1 1 1 0 -1 0 0 >= 0
 1 0 1 0 1 0 1 -1 0 0 >= 0
 1 0 1 0 1 1 0 -1 0 0 >= 0
 1 0 1 1 0 1 1 -1 0 0 >= 0
 1 1 0 0 1 0 1 -1 0 0 >= 0
 1 1 0 0 1 1 0 -1 0 0 >= 0 
 1 1 0 1 0 1 1 -1 0 0 >= 0
 0 0 0 0 0 0 0 1 0 -1 >= 0
 1 0 0 0 0 0 0 0 0 0 >= 0
 0 1 0 0 0 0 0 0 0 0 >= 0
 0 0 1 0 0 0 0 0 0 0 >= 0
 0 0 0 1 0 0 0 0 0 0 >= 0
 0 0 0 0 1 0 0 0 0 0 >= 0
 0 0 0 0 0 1 0 0 0 0 >= 0
 0 0 0 0 0 0 1 0 0 0 >= 0
 0 0 0 0 0 0 0 0 0 1 >= 0
 1 1 1 1 1 1 1 0 -2 1 <= 0
VerticesFloat
ExtremeRays
VerticesOfPolyhedron
\end{verbatim}
\quad Running \textit{Normaliz} for all possible choices of $a_1\cdots
a_{s+1}$, $s=7$, $a_{s+1}\notin\{0,1\}$, by
Theorem~\ref{lp-resurgence-formula}, we obtain that 
$\rho_{ic}(I)=4/3=1.33333$. The optimal value of the linear programs of
Theorem~\ref{lp-resurgence-formula} is attained at the vertex
$$
(2/3,\,2/3,\,2/3,\,0,\,0,\,0,\,0,\,4/3,\,1,\,0)
$$
of the polyhedron of feasible points that corresponds to 
$1\ 1\ 1\ 1\ 1\ 1\ 1\,-2$.  
\end{procedure}

\begin{algorithm}\label{AS-code}
The function ``rhoichypes'' in the following procedure for
\textit{Macaulay}$2$ \cite{mac2} computes the asymptotic resurgence of
the ideal $I$ of Example~\ref{bowtie-example} using the algorithm described in
Procedure~\ref{bowtie-procedure}. The input for this function
is the matrix whose rows are the generators of the Rees
cone of $I$. This procedure uses the interface of \textit{Macaulay}$2$
to \textit{Normaliz} \cite{normaliz2}. To compute other examples, in
the next procedure simply change the polynomial ring in line $4$ and the matrix starting
at line $5$. 

\begin{verbatim}
restart
loadPackage("Normaliz",Reload=>true)
loadPackage("Polyhedra", Reload => true)
 R=QQ[x_1..x_8];
l={{1, 0, 0, 0, 0, 0, 0, 0},
   {0, 1, 0, 0, 0, 0, 0, 0},
   {0, 0, 1, 0, 0, 0, 0, 0},
   {0, 0, 0, 1, 0, 0, 0, 0},
   {0, 0, 0, 0, 1, 0, 0, 0},
   {0, 0, 0, 0, 0, 1, 0, 0},
   {0, 0, 0, 0, 0, 0, 1, 0},
   {1, 1, 0, 0, 0, 0, 0, 1},
   {0, 1, 1, 0, 0, 0, 0, 1},
   {1, 0, 1, 0, 0, 0, 0, 1},
   {1, 0, 0, 1, 0, 0, 0, 1},
   {0, 0, 0, 1, 1, 0, 0, 1},
   {0, 0, 0, 0, 1, 1, 0, 1},
   {0, 0, 0, 0, 0, 1, 1, 1},
   {0, 0, 0, 0, 1, 0, 1, 1}}
L=for i in l list R_i
nmzFilename="rproj1";
intclToricRing L;
--The support hyperplanes of the Rees cone of I are:
hypes=readNmzData("sup")
rhoichypes = hypes -> (
choices = select (entries hypes, l-> not 
isSubset({last l}, {0,1}));
possibilities = for i to #choices-1 list
( 
l = choices_i;
l' = apply( drop(l,-1)|{0, last(l), 1}, a-> a_ZZ);
b  = {0};
s = select (entries hypes, l-> isSubset({last l}, 
set {0, -1}));
s' = apply(s, l -> -l |{0,0});
s' = apply(s', a-> apply(a, c-> c_ZZ));
b = b | for i to #s'-1 list 0;
v = for i to #l-2 list 0;
t = { v | {-1, 0, 1}, v | {0, 0, -1} };
t' = { v | {0, 1, 0} };
t = apply(t, a-> apply(a, c-> c_ZZ));
b = b | { 0, 0 };
b = apply(b, c-> c_ZZ);
A = matrix({l'} | s' | t);
b = transpose matrix{b};
C = matrix t';
d = matrix{{1}};
P = polyhedronFromHData( A, b, C, d);
vert = vertices P;
max flatten entries vert^{numrows vert -3});
return max possibilities
)
--This gives the ic-resurgence of I which is equal to 
--the asymptotic resurgence of I
rhoichypes hypes
\end{verbatim}
\end{algorithm}

\end{appendix}


\section*{Acknowledgments} 
Computations with \textit{Normaliz} \cite{normaliz2}, \textit{PORTA} \cite{porta}, 
and \textit{Macaulay}$2$  were important to gain a better understanding of the monomial
ideals and algebras of filtrations 
of covering polyhedra. 

\bibliographystyle{plain}

\begin{thebibliography}{10}

\bibitem{Alilooee-Benerjee} A. Alilooee and A. Banerjee, 
Packing properties of cubic square-free monomial ideals, 
J. Algebraic Combin. (2021).
\url{https://doi.org/10.1007/s10801-021-01020-2}.

\bibitem{berge-graphs-hypergraphs} C. Berge, {\it Graphs and
Hypergraphs}, North-Holland Mathematical Library, Vol. 6,
North-Holland Publishing Co., 
Amsterdam-London; American Elsevier Publishing Co., Inc., New York,
1976. 

\bibitem{bertsimas} D. Bertsimas and J. N. Tsitsiklis, {\it Introduction 
to Linear Optimization\/}, Athena Scientific, Massachusetts, 1997.

\bibitem{Waldschmidt-Bocci-etal} C. Bocci, S. Cooper, E. Guardo, B.
Harbourne, M. Janssen, U. Nagel, A. Seceleanu, A. Van Tuyl and T. Vu,
The Waldschmidt constant for squarefree monomial ideals, J. Algebraic
Combin. {\bf 44} (2016), no. 4, 875--904.

\bibitem{resurgence} C. Bocci and B. Harbourne, Comparing powers and symbolic powers of
ideals, J. Algebraic Geom. {\bf 19} (2010), no. 3, 399--417.


\bibitem{Seceleanu-packing} H. Bodas, B. Drabkin, C. Fong, S. Jin, J. Kim, W. Li, A. 
Seceleanu, T. Tang and B. Williams, Consequences of the packing
problem, J. Algebraic Combin. (2021).
\url{https://doi.org/10.1007/s10801-021-01039-5}.

\bibitem{boyd}	S. Boyd and L. Vandenberghe, \textit{Convex
Optimization}, Cambridge University Press, Cambridge, 2004. 

\bibitem{normaliz2} W. Bruns, B. Ichim, T. R\"omer, R. Sieg and C. S\"oger:
Normaliz. Algorithms for rational cones and affine monoids.
Available at \url{https://normaliz.uos.de}.

\bibitem{Seceleanu-convex-bodies} 
J. Camarneiro, B. Drabkin, D. Fragoso, W. Frendreiss, D. Hoffman, 
A. Seceleanu, T. Tang and S. Yang, Convex bodies and asymptotic
invariants for powers of monomial ideals. 
Preprint 2021, \url{https://arxiv.org/abs/2101.04008}.

\bibitem{porta} T. Christof, revised by A. L\"{o}bel and M. Stoer,
   {\em PORTA\/}: A Polyhedron Representation
Transformation Algorithm, 1997.  
\url{https://porta.zib.de/}.

\bibitem{Cooper-symbolic} S. M. Cooper, R. J. D. Embree, H. T. H$\rm
\grave{a}$ and A. H. Hoefel, 
Symbolic powers of monomial ideals. Proc. Edinb. Math. Soc. {\bf (2)
60} (2017), no. 1, 39--55. 


\bibitem{cornu-book}{G. Cornu\'ejols, {\it Combinatorial optimization{\rm:} 
Packing and covering\/}, CBMS-NSF Regional Conference Series in Applied 
Mathematics {\bf 74}, SIAM (2001).}

\bibitem{symbolic-powers-survey} H. Dao, A. De Stefani, E. Grifo, C. Huneke and L.
N\'u\~nez-Betancourt, Symbolic powers of ideals, in {\it Singularities and
Foliations. Geometry, Topology and Applications.\/} (R.
Ara\'ujo dos Santos, A. Menegon Neto, D. Mond, M. Saia and J.
Snoussi, Eds.), Springer Proceedings in Mathematics \& Statistics, vol. 222, 
Springer, 2018, pp. 387--432. 

\bibitem{mc8} D. Delfino, A. Taylor, W. V. Vasconcelos, N. Weininger and
R. H. Villarreal, Monomial ideals and the computation of
multiplicities, \textit{Commutative ring theory and applications} 
(Fez, 2001), Lect. Notes Pure Appl. Math. {\bf 231}
(2003), 87--106,  
Dekker, New York, 2003. 

\bibitem{diestel}{R. Diestel, \textit{Graph Theory}, 
Graduate Texts in  Mathematics 
{\bf 173}, Springer-Verlag, New York, 1997.}

\bibitem{DiPasquale-Drabkin}
M. DiPasquale and B. Drabkin, 
On resurgence via asymptotic resurgence, J. Algebra (2021).   
\url{https://doi.org/10.1016/j.jalgebra.2021.07.021}.

\bibitem{Francisco-TAMS} M. DiPasquale, C. A. Francisco, J. Mermin and
J. Schweig, Asymptotic resurgence via integral closures, 
Trans. Amer. Math. Soc. {\bf 372} (2019), no. 9, 6655--6676. 

\bibitem{mfmc} L. A. Dupont and R. H. Villarreal, 
Algebraic and combinatorial properties of ideals and 
algebras of uniform clutters of TDI systems, 
J. Comb. Optim. {\bf 21} (2011), no. 3, 269--292.

\bibitem{Eisen}{D. Eisenbud, {\it Commutative Algebra with a view
toward Algebraic Geometry\/}, Graduate
Texts in  Mathematics {\bf 150}, Springer-Verlag, 1995.}

\bibitem{normali} C. Escobar, R. H. Villarreal and Y. Yoshino, Torsion
freeness and normality of blowup rings of monomial ideals, 
{\it Commutative Algebra\/}, Lect. Notes Pure Appl. Math. 
{\bf 244}, Chapman \& Hall/CRC, Boca Raton, FL, 2006, pp. 69--84. 

\bibitem{reesclu}{I. Gitler, E. Reyes and R. H. Villarreal, 
Blowup algebras of square--free monomial ideals and some links to
combinatorial optimization problems, 
Rocky Mountain J. Math. {\bf 39} (2009), no. 1, 71--102.} 

\bibitem{clutters} I. Gitler, C. Valencia and R. H. Villarreal, 
A note on Rees algebras and the MFMC property, Beitr\"age Algebra
Geom. {\bf 48} 
(2007), no. 1, 141--150.

\bibitem{cm-oriented-trees} P. Gimenez, J. Mart\'\i nez-Bernal, A.
Simis, R. H. Villarreal and C. E. Vivares, Symbolic powers of monomial ideals and Cohen-Macaulay
vertex-weighted digraphs, in {\it Singularities, Algebraic Geometry,
Commutative Algebra, and Related Topics\/} (G. M. Greuel, et.al.
Eds), Springer, Cham, 2018, pp. 491--510.

\bibitem{GoNi}{S. Goto and K. Nishida, The Cohen--Macaulay and 
Gorenstein Rees algebras associated to filtrations, 
{\it Mem. Amer. Math. Soc.} {\bf 110} (1994), no. 526, 1--134.}

\bibitem{mac2} D. Grayson and M. Stillman, 
{\em Macaulay\/}$2$, 1996. Available at
\url{http://www.math.uiuc.edu/Macaulay2/}.

\bibitem{asymptotic-resurgence} E. Guardo, B. Harbourne and A. Van
Tuyl, Asymptotic resurgences for
ideals of positive dimensional subschemes of projective space, Adv.
Math. {\bf 246} (2013), 114--127.

\bibitem{HaM} H. T. H$\rm \grave{a}$ and S. Morey, 
Embedded associated
primes of powers of square-free monomial ideals, 
J. Pure Appl. Algebra {\bf 214} (2010), no. 4, 301--308.

\bibitem{Herzog-Hibi-book} J. Herzog and T. Hibi, {\it Monomial
Ideals}, 
Graduate
Texts in  Mathematics {\bf 260}, Springer-Verlag, 2011.

\bibitem{cover-algebras} J. Herzog, T. Hibi and N. V. Trung, 
Symbolic powers of monomial ideals and vertex cover algebras, 
Adv. Math. {\bf 210} (2007), 304--322.  

\bibitem{hhtz} J. Herzog, T. Hibi, N. V. Trung and X. Zheng, Standard
graded vertex cover algebras, cycles and leaves, 
{\it Trans. Amer. Math. Soc.} {\bf 360} (2008), 6231--6249. 

\bibitem{huneke-swanson-book} 
C. Huneke and I. Swanson, {\it Integral Closure of Ideals Rings, and
Modules}, London Math. Soc., Lecture Note Series {\bf 336}, Cambridge
University Press, Cambridge, 2006.

\bibitem{korte} B. Korte and J. Vygen, \textit{Combinatorial
optimization Theory and algorithms}, Third edition, Algorithms and
Combinatorics {\bf 21}, Springer-Verlag, Berlin, 2006. 

\bibitem{MNB} J. Monta\~no and L. N\'u\~nez-Betancourt, 
Splittings and symbolic powers of square-free monomial ideals, 
Int. Math. Res. Not. IMRN 2021, no. 3, 2304--2320. 

 \bibitem{WOG} Y. Pitones, E. Reyes and J. Toledo, 
Monomial ideals of weighted oriented graphs, Electron. J. Combin.
{\bf 26} (2019), no. 3, Paper 44, 18 pp.  

\bibitem{Scheinerman} E. R. Scheinerman and D. H. Ullman, 
\textit{Fractional graph theory}, a rational approach to the
theory of graphs, John Wiley $\&$ Sons, Inc., New York, 1997. 

\bibitem{SchenzelFiltrations}
P. Schenzel, Filtrations and Noetherian symbolic blow-up rings, 
Proc. Amer. Math. Soc. {\bf 102} (1988), no. 4, 817--822.


\bibitem{Schr1}{A. Schrijver, On total dual integrality, 
Linear Algebra Appl. {\bf 38} (1981), 27--32.}

\bibitem{Schr}{A. Schrijver, {\it Theory of Linear and Integer
Programming\/}, John Wiley \& Sons, New York, 1986.}

\bibitem{Schr2} {A. Schrijver, {\it Combinatorial Optimization\/}, 
Algorithms and Combinatorics {\bf 24}, Springer-Verlag, Berlin, 2003.}

\bibitem{Trung} N. V. Trung, Integral closures of monomial ideals and
Fulkersonian hypergraphs, Vietnam J. Math. {\bf 34} (2006), no. 4,
489--494. 

\bibitem{Vas1}{W. V. Vasconcelos, {\it Computational Methods in
Commutative Algebra and Algebraic Geometry\/}, 
Springer-Verlag, 1998.}

\bibitem{bookthree} W. V. Vasconcelos, {\it Integral Closure\/},
Springer Monographs in Mathematics, Springer-Verlag, New York, 2005.


\bibitem{perfect} R. H. Villarreal, Rees algebras and polyhedral cones of ideals 
of vertex covers of  perfect graphs, J. Algebraic Combin. {\bf 27}
(2008), 293--305.

\bibitem{monalg-rev} R. H. Villarreal, {\it Monomial Algebras\/},
Second edition, 
Monographs and Research Notes in Mathematics, Chapman and Hall/CRC,
Boca Raton, FL, 2015.

\bibitem{webster}{R. Webster, {\it Convexity\/}, Oxford 
University Press, Oxford, 1994.}

\bibitem{ZS}{O. Zariski and P. Samuel, 
{\em Commutative Algebra\/}, Vol. II, Van Nostrand, 
Princeton, 1960.}

\end{thebibliography}

\end{document}